\documentclass[11pt]{amsart}
\usepackage{graphicx}
\usepackage{mathpazo}
\usepackage{amssymb}
\usepackage{epstopdf}
\usepackage[cmtip,all]{xy}
\usepackage{hyperref}
\usepackage{mathrsfs}
\usepackage{enumerate}
\hypersetup{linktocpage}

\newtheorem{prop}{Proposition}[section]
\newtheorem{lemma}[prop]{Lemma}
\newtheorem{thm}[prop]{Theorem}
\newtheorem{rem}[prop]{Remark}

\usepackage{mathrsfs}
\usepackage{stmaryrd}
\usepackage{amsthm}
\usepackage{amssymb}
\usepackage{amsfonts}
\usepackage{tikz} 
\usepackage{mathtools}
\usepackage{setspace}
\usepackage{cancel}
\usepackage{chemarrow}
\usetikzlibrary{matrix,arrows} 
\usepackage{amscd}

\newtheorem{theorem}{Theorem}
\newtheorem{corollary}[theorem]{Corollary}

\newtheorem{cor}[prop]{Corollary}

\theoremstyle{definition}
\newtheorem{example}[prop]{Example}

\newtheorem{defn}[prop]{Definition}

\newcommand{\CC}{\mathbb{C}}
\renewcommand{\P}{\mathbb{P}}
\renewcommand{\L}{\mathcal{L}}

\newcommand{\I}{\mathcal{I}}

\renewcommand{\O}{\mathcal{O}}
\newcommand{\Z}{\mathbb{Z}}
\newcommand{\Q}{\mathbb{Q}}
\newcommand{\F}{\mathcal{F}}

\newcommand{\cP}{\mathcal{P}}
\newcommand{\G}{\mathcal{G}}

\newcommand{\M}{\mathcal{M}}

\newcommand{\hilb}{\operatorname{Hilb}}

\newcommand{\Ex}{\operatorname{Ext}}

\newcommand{\Hom}{\operatorname{Hom}}
\newcommand{\ch}{\operatorname{ch}}

\newcommand{\ext}{\operatorname{Ext}}
\newcommand{\DT}{\overline{DT}}
\newcommand{\T}{{\bf{T}}}

\newcommand{\s}{\includegraphics[width=.13in]{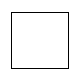}}
\newcommand{\rs}{\includegraphics[width=.21in]{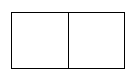}}
\newcommand{\us}{\includegraphics[width=.13in]{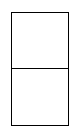}}
\newcommand{\rus}{\includegraphics[width=.22in]{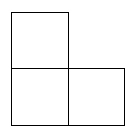}}
\newcommand{\uss}{\includegraphics[width=.13in]{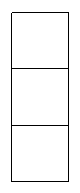}}
\newcommand{\rss}{\includegraphics[width=.31in]{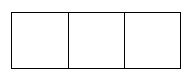}}
\newcommand{\cc}{\mathsf{c}}

\title[Invariants of 2-dimensional sheaves on local $\P^2$]{Generalized Donaldson-Thomas Invariants of 2-Dimensional sheaves on local $\P^2$}
\author[Amin Gholampour and Artan Sheshmani]{Amin Gholampour and Artan Sheshmani}
\date{\today}                                   
\begin{document}
\maketitle
\begin{abstract}
Let $X$ be the total space of the canonical bundle of $\P^2$. We study the generalized Donaldson-Thomas invariants defined in \cite{a30} of the moduli spaces of the 2-dimensional Gieseker semistable sheaves on $X$ with first Chern class equal to $k$ times the class of the zero section of $X$. When $k=1,2,$ or $3$, and semistability implies stability, we express the invariants in terms of known modular forms. We prove a combinatorial formula for the invariants when $k=2$ in the presence of the strictly semistable sheaves, and verify the BPS integrality conjecture of \cite{a30} in some cases. 
\end{abstract}

\setcounter{section}{0}
\section{Introduction}
The study of abelian gauge theory on $\mathbb{R}^4$ led string theorists to discover an interesting symmetry, the electric-magnetic duality, which inverts the coupling constant and extends to an action of $SL_{2}(\mathbb{Z})$. This $SL_{2}(\mathbb{Z})$ symmetry was later studied over the more general 4-manifolds, some with specific topological twists, where it was called $S$-duality , say in the context of $N=2,4$ supersymmetric Yang-Mills theories. The $S$-duality is roughly saying that certain 4-dimensional gauge theories are modular invariant under the action of the $SL_{2}(\Z)$. This modular invariance can be tested by studying the partition function of the theory, roughly speaking, measuring the Euler characteristic of the instanton moduli space of that theory. One of interesting results by string theorists in \cite{a136} was to show that the topological quantum field theories obtained by the so-called (topologically) twisted super Yang-Mills theories over the four manifold  and their associated partition functions are equivalent to Donaldson's theory of four manifolds and the associated partition function of Donaldson's invariants. These interesting consequences of S-duality, later set base for much further developments, such as the correspondence between the supersymmetric black hole entropies and $N=2,d=4$ super Yang-Mills theories \cite{a98}. Recently the study of the conjectural modular properties of the BPS partition functions of the supersymmetric D4-D2-D0 black holes \cite{a78}, \cite{a95} motivated algebraic geometers to construct a mathematical framework for modeling the D4-D2-D0 BPS states and prove the modularity properties of their associated partition functions, using purely algebraic-geometric techniques. The current article is the third in the sequel, after \cite{a123} and \cite{a144}, of the authors' attempt to achieve this goal where here, the focus of the study is to specifically compute the D4-D2-D0 BPS invariants via computing their mathematical counterpart given by the Donaldson-Thomas invariants of torsion sheaves in an ambient Calabi-Yau threefold. 

Let $(X,\O(1))$ be a smooth projective Calabi-Yau threefold. For a pure sheaf $\F$ the Hilbert polynomial is defined to be $P_{\F}(m)=\chi(X,\F(m))$, and the reduced Hilbert polynomial of $\F$ is $$p_\F=\frac{P_\F}{\text{leading coefficient of }P_\F}=m^d+a_\F m^{d-1}+\cdots.$$ 
\begin{itemize}
\item $\F$ is called Gieseker semistable if for any proper subsheaf $\G \subset \F$ we have $p_\G(m) \le p_\F(m)$ for $m\gg 0$. $\F$ is called Gieseker stable if the equality never holds for any proper subsheaf $\G$. 
\item $\F$ is called $\mu$-semistable if for any proper subsheaf $\G \subset \F$ we have $a_\G \le a_\F$. $\F$ is called $\mu$-stable if the equality never holds for any proper subsheaf $\G$. 
\item For a fixed $n\gg 0$, a pair $(\F,s)$, where $s$ is a nonzero section of $\F(n)$, is called stable (\cite[Section 12]{a30}) if
\begin{enumerate}
\item $\F$ is Gieseker semistable, 
\item if $s$ factors through a proper subsheaf $\G(n)$ then $p_\G(m)<p_\F(m)$ for $m\gg 0$. \hfill $\oslash$
\end{enumerate}
\end{itemize}

The stability of pairs has originated from the stability of the coherent systems defined by Le Potier \cite{a12}. The reduced Hilbert polynomial of the pair $(\F,s)$ is defined to be 
\begin{equation}\label{delta-stab}
p_{(\F,s)}=\frac{P_\F+\delta_{0s}\epsilon}{\text{leading coefficient of } P_{\F}}
\end{equation}
 where and $0<\epsilon \ll 1$. Now if $\G \subset \F$, define $s'$ to be the restriction of $s$ if $s$ factors through $\G(n)$, and otherwise define $s'=0$. Now the pair $(\F,s)$ is stable if and only if for any proper subsheaf $\G \subset \F$, $p_{(\G,s')}<p_{(\F,s)}.$

Suppose now that $\M=\M(X;P)$ is a proper moduli space of Gieseker stable sheaves $\F$ (or the moduli space of stable pairs $(\F,s)$) as above with fixed Hilbert polynomial $P_\F=P$. The moduli space $\M$ is usually singular and may have several components with different dimensions. To define (deformation invariant $\Z$-valued) invariants $DT(X;P)$ as \emph{integration} over $\M$ we need to have a \emph{virtual fundamental class} of the moduli space constructed by means of a \emph{perfect obstruction theory} on $\M$. This can be obtained by studying the deformations and obstructions of the stable sheaves or the stable pairs \cite{a2,a14,a20,a10}. 
Moreover, the obstruction theory on $\M$ is symmetric and the corresponding invariants are expressible as a weighted Euler characteristic of the moduli space \cite{a1}. \\

If the moduli space of Gieseker semistable sheaves $\M=\M(X;P)$ contains strictly semistable sheaves, then one cannot define the invariants $DT(X;P)$ by means of the virtual fundamental class. 
Joyce and Song \cite{a30} instead  define the $\Q$-valued invariants for $\M$ called the \emph{generalized DT} invariants $\DT(X;P)$ which are given by the ``\textit{stacky}'' weighted Euler characteristic of the moduli space of semistable sheaves. Joyce-Song stable pairs theory \cite{a30} provides a tool to compute the so-called stacky Euler characteristics, by using the sophisticated motivic wall-crossing techniques developed by Joyce \cite{a58}, as well as Kontsevich and Soibelman in \cite{a42}. In other words, the main idea is to benefit from the, easier to calculate, Joyce Song pair invariants to compute the generalized DT invariants. The latter can be done by obtaining a wall-crossing identity between the the elements of the Hall algebra of the motivic ring of stack functions of the moduli space of stable pairs and the moduli space of semistable sheaves respectively. After taking the stacky Euler characteristics of both sides of this identity, one obtains the wall-crossing identity between the pair invariants and the generalized DT invariants. Note that, $\DT(X;P)$ specializes to $DT(X;P)$ if there are no strictly semistable sheaves and moreover, $\DT(X;P)$ is also deformation invariant.

We study the case where $X$ is the total space of the canonical bundle of $\P^2$ and $\M(X;P)$ is the moduli space of semistable sheaves with Hilbert polynomial $P(m)=rm^2/2+\dots$. Any semistable sheaf $\F$ with Hilbert polynomial $P$ is (at least set theoretically) supported on the zero section of $X$, and $c_1(\F)$ is equal to $r$ times the class of the zero section. We relate $\DT(X;P)$ to the topological invariants of the moduli space of torsion-free semistable sheaves on $\P^2$. Using the wall-crossing formula of Joyce-Song \cite{a30} and the toric methods of \cite{a115, a34} we find a formula for $\DT(X;P)$ when $r=2$ in the presence of strictly semistable sheaves. To express the main result, let $\M(\P^2;P)$ be the moduli space of rank 2 Gieseker semistable sheaves on $\P^2$ with Hilbert polynomial $P$ and let $\M^s(\P^2;P)$ be the open subset of stable sheaves. Denote by $\hilb^{n}(\P^2)$ the Hilbert scheme of $n$ points on $\P^2$. Then we prove

\begin{theorem} \label{thm:P2}
Let $P(m)=m^2+3m+2+b$ where $0\ge b \in \Z$, \begin{enumerate} \item If $b$ is an odd number then $\DT(X;P)=DT(X;P)=\chi(\M(\P^2;P))$. \item If $b$ is an even number then 
\begin{align*}
\DT(X;P)&=\chi(\hilb^{-b/2}(\P^2))/4-\chi(\M^s(\P^2;P))-\frac{1}{2}\cc^{ss}(b), 
\end{align*}
where $\cc^{ss}(b)\in \Z$ is a combinatorial expression (cf. Theorem \ref{thm:ss}) taking into account the contribution of indecomposable strictly semistable sheaves.
\end{enumerate}
\end{theorem}

$\DT(X;P)$ is in general a rational number in the presence of semistable sheaves. Joyce and Song in \cite[Section 6.2]{a30} define the corresponding BPS invariants denoted by $\hat{DT}(X;P)$ by the following formula:
$$\DT(X;P)=\sum_{d \ge 1,\; d \mid P(m)}\frac{1}{d^2}\hat{DT}(X;P/d).$$ Joyce and Song conjecture that $\hat{DT}(X;P)$ is an integer. In the case that there are no strictly semistable sheaves with Hilbert polynomial $P$ we have $\hat{DT}(X;P)=DT(X;P)$.

\begin{corollary} \label{cor:ss}Using the notation of Theorem \ref{thm:ss}, we assume that $b$ is an even number then \footnote{By this result to show $\hat{DT}(X;P) \in \Z$ one needs to show that $\cc^{ss}(b) \in 2\Z.$ }
\begin{align*}
\hat{DT}(X;P)=-\chi(\M^s(\P^2;P))-\frac{1}{2}\cc^{ss}(b).
\end{align*}
\end{corollary} \qed

\section*{Acknowledgment}
We would like to thank Martijn Kool, Jan Manschot and Richard Thomas for many helpful discussions. The second author would like to thank Max Planck Institut f\"{u}r Mathemaik for hospitality.

\section{Proof of Theorem \ref{thm:P2}}\label{sec:P2}

Let $X$ be the total space of $\O(-3)$ over $\P^2$. Then $X$ is a quasiprojective Calabi-Yau threefold, called local $\P^2$. Let $L$ be the pullback of $\O(1)$ from $\P^2$, and let $S\cong \P^2 \subset X$ denote the zero section. We identify the compactly supported cohomology groups of $X$ with the cohomology groups of $\P^2$: $$H^{*}_{cs}(X,\Q)\cong H^{*-2}(\P^2,\Q).$$ Using this identification, let $H\in H^2_{cs}(X,\Q)$, $h\in H^4_{cs}(X,\Q)$, $pt\in H^6_{cs}(X,\Q)$ be respectively the classes of $S$, a line and a point on $S$.  
The Hilbert polynomial (with respect to $L$) of a 2-dimensional compactly supported coherent sheaf $\F$ on $X$ with the compactly supported Chern character $$\ch_{cs}(\F)=\big(0,kH,(3k/2+a) h,(3k/2+3a/2+b)pt\big)\in \oplus_{i=0}^3 H^{2i}_{cs}(X,\Q)$$ is given by $$P(m)= \frac{k}{2}m^2+\frac{3k+2a}{2}m+\frac{3k+2a+2b}{2}.$$ Any such $\F$ is set theoretically supported on $S$. Moreover, we have
\begin{lemma} \label{lem:support}
If $\F$ as above is semistable \footnote{Whenever we mention (semi)\-stability of sheaves, unless otherwise is specified, we always mean Gieseker (semi)\-stability.
}, then $\F$ is scheme theoretically supported on $S$ and hence $$\M(X; P)\cong \M(\P^2; P),$$ the moduli space of rank $k$ semistable sheaves on $\P^2$ with Hilbert polynomial $P$. 
\end{lemma}
\begin{proof}
The ideal sheaf of $S$ in $X$ is isomorphic to $L^3$, hence we get the exact sequence $$\F\otimes L^3 \to \F \to \F|_S\to 0.$$Since $\F$ is semistable, the first morphism in the sequence above is necessarily zero and hence $\F\cong \F|_S $. 
\end{proof}
Note that for any stable torsion-free sheaf $\F$ on $\P^2$ we have $\Ex^2(\F,\F)=0$ by Serre duality and the negativity of $K_{\P^2}\cong \O(-3)$. Therefore, if $P(m)$ is such that there are no strictly semistable sheaves on $\P^2$ with Hilbert polynomial $P(m)$, then the moduli space $\M=\M(X;P)\cong \M(\P^2;P)$ is un-obstructed and smooth of dimension $$\dim \M= 1- \chi(\F,\F)=-2kb+a^2-k^2+1,$$ where as a sheaf on $\P^2$ $$\ch_0(\F)=k,\quad \ch_1(\F)=a\cdot h, \quad \ch_2(\F)=b\cdot pt.$$  In this case the Behrend's function is determined by $\nu_\M=(-1)^{\dim \M}$ \cite{a1}, and hence $$DT(X;P)=(-1)^{\dim \M}\chi(\M).$$ The generating function of the Euler characteristic of the moduli space of $\mu$-stable torsion-free sheaves on $\P^2$ is known for $k=1,2,3$, by the results of \cite{a110, a119, a89, a135, a136} and they all have modular properties. Here is the summary of these results: 
\begin{enumerate}
\item $k=1$. By tensoring with $\O(-a)$ we can assume that $a=0$. So then $\M(X;P)\cong \hilb^{-b}(\P^2)$, the Hilbert scheme of $-b$ points on $\P^2$, which is smooth of dimension $-2b$. Note that in this case there are no strictly semistable sheaves on $\P^2$ with Hilbert polynomial $m^2+3m/2+b+1$, so by \cite{a89}
\begin{equation} \label{ech}  \sum_b DT(X;m^2/2+3m/2+b+1)q^{-b}=\prod_{n> 0}\frac{1}{(1-q^n)^3}.\end{equation} 
\item $k=2.$ By tensoring with $\O(\lfloor-a/2 \rfloor)$ we can assume that either $a=0$ or $a=1$. If $a=1$ then there are no strictly semistable sheaves with the corresponding Hilbert polynomial $m^2+4m+7/2+b$ (and $\mu$-stability is equivalent to Gieseker stability) and hence $\M(X;m^2+4m+7/2+b)$ is smooth of dimension $-4b-2$ so by \cite[Corollary 4.2]{a110} 
\begin{align*}&\sum_{\tiny \begin{array}{c} \tiny b \in (1/2)\Z \\ b \neq \lfloor b\rfloor \end{array}} DT(X;m^2+4m+7/2+b)q^{1/2-b}\\&=\chi(\M(X;m^2+4m+7/2+b))\\&=\frac{1}{\prod_{n\ge 0}(1-q^n)^6}\sum_{m=1}^\infty\sum_{n=1}^\infty\frac{q^{mn}}{1-q^{m+n-1}}.\end{align*} 
When $a=0$, there are strictly semistable sheaves with Hilbert polynomial $m^2+3m+2+b$ only when $b\in 2\Z$.
If $b \not \in 2\Z$ then $\M$ is smooth of dimension $-4b-3$. So for $b=2b'+1$
$$DT(X;m^2+3m+3+2b')= -\chi(\M(\P^2;m^2+3m+3+2b')).$$  
We will study the case $b\in 2\Z$ in more detail in what follows in this section.
\item $k=3.$ We can assume again that $a=0,1$ or $2$. In the latter two cases $\M$ has no strictly semistable sheaves (and $\mu$-stability is equivalent to Gieseker stability) and there is a modular formula for the generating function $DT(X;P)$ in terms of the generating function of the Euler characteristics of $\M(\P^2;P)$ computed in \cite[Section 4.3]{a110}. 
\end{enumerate}

In the following we compute $\DT(X;P)$ in the presence of semistable sheaves when $k=2$. By the discussion above strictly semistable sheaves only occur if $a=0$ and $b\in 2\Z$. Let 
\begin{equation}\label{Hilb-poly}
P(m)=m^2+3m+2+b \quad \quad b\in 2\Z
\end{equation} 
be the corresponding Hilbert polynomial. We use the moduli space of stable pairs in the sense of \cite{a30}. 

For $n\gg 0$, let $\cP_n=\cP_n(X;P)$ be the moduli space of stable pairs $(\F,s)$ where $\F$ is a semistable sheaf of rank 2 with Hilbert polynomial \eqref{Hilb-poly}, and $s$ is a nonzero section of $\F(n)$. The stability of pairs further requires that if $\G \neq 0$ is a proper subsheaf of $\F$, such that $s$ factors through $\G(n)$, then the Hilbert polynomial of $\G$ is strictly less than the Hilbert polynomial of $\F$. By \cite{a30} $\cP_n$ admits a symmetric perfect obstruction theory. Let $PI_n=PI_n(X;P)$ be the corresponding pair invariants. Note that, even though $X$ is not proper, $\cP_n$ is proper (as all the semistable sheaves are supported on $\P^2\cong S \subset X$) so $PI_n$ is well defined. Alternatively, $PI_n=\chi(\cP_n,\nu_{\cP_n})$.

\begin{lemma} \label{lem:wall-crossing}
$\overline{DT}(X;P)=DT(X;P/2)^2\cdot P(n)/8-PI_n(X;P)/P(n).$
\end{lemma}
\begin{proof}
This is a direct corollary of the wall-crossing formula \cite[5.17]{a30} by noting two facts. Firstly, the only decomposable semistable sheaves with Hilbert polynomial $P$ are of the form  $\I_{Z_1}\oplus \I_{Z_2}$ where $\I_{Z_1}$ and $\I_{Z_2}$ are the push forwards to $X$ of the ideal sheaves of the 0-dimensional subschemes $Z_1, Z_2 \subset \P^2$ of length $-b/2$. Secondly, the Euler form $\chi(\I_{Z_1},\I_{Z_2})=0$.
\end{proof}

\begin{rem}
Note the polynomial on the right hand side of Lemma \ref{lem:wall-crossing} is a rational number independent of $n\gg 0$.
\hfill $\oslash$

\end{rem}

There is a natural morphism $\cP_n \to \M$ that sends a stable pair $(\F,s)$ to the $S$-equivalence class of $\F$. Note that $\M$ is singular at a point corresponding to a strictly semistable sheaf. However, we have

\begin{prop} \label{lem:pair-smooth} 
$\cP_n(X;P)$ is a smooth scheme of dimension $P(n)-4b-4$.
\end{prop}
\begin{proof}
We denote by $I^\bullet$ the 2-term complex $\O(-n) \xrightarrow{s(-n)} \F$ corresponding to a stable pair $(\F,s)$. By the stability of pairs $\F$ has to be a semistable sheaf and hence Lemma \ref{lem:support} implies that $$\cP_n(X;P)\cong \cP_n(\P^2;P),$$ the moduli space of the stable pairs on $\P^2$. The Zariski tangent space and the obstruction space at a $\CC$-point $(F,s) \in \cP_n$ are then identified with $\Hom_{\P^2}(I^\bullet,\F)$ and $\ext^1_{\P^2}(I^\bullet,\F)$ respectively. Consider the following natural exact sequence:
\begin{align*}
&
0\rightarrow \Hom_{\P^2}(\F,\F)\rightarrow \Hom_{\P^2}(\O(-n),\F)\rightarrow \Hom_{\P^2}(I^{\bullet},\F)\rightarrow \ext^{1}_{\P^2}(\F,\F)\\
&
\rightarrow \ext^{1}_{\P^2}(\O(-n),\F)\to \ext^1_{\P^2}(I^\bullet,\F) \to \ext^2_{\P^2}(\F,\F)
\end{align*}
Since $n \gg 0$, we have $\ext^{i}_{\P^2}(\O(-n),\F)\cong H^i(\P^2,\F(n))=0$ for $i>0$. 

We also know that $\ext^2_{\P^2}(\F,\F)=0$ by Serre duality and the semistability of $\F$. So the exact sequence above firstly implies that $\ext^1(I^\bullet,\F)=0$ which means that $\cP_n$ is unobstructed and hence smooth, and secondly $$\dim \Hom(I^{\bullet},F)=\chi_{\P^2}(\F(n))-\chi_{\P^2}(\F,\F)=P(n)-4b-4.$$
 \end{proof}
By Proposition \ref{lem:pair-smooth}, and noting that $P(n) \in 2\Z$ (see \eqref{Hilb-poly}), we have 
\begin{cor} \label{lem:PI}
 $PI_n(X;P)=\chi(\cP_n(X;P))=\chi(\cP_n(\P^2;P))$. 
 \qed 
\end{cor} 
We will find $\chi(\cP_n(\P^2;P))$ using toric techniques. According to \cite{a110}, a torsion-free $\T=\CC^{*2}$-equivariant sheaf $\F$ on $\P^2$ corresponds to three compatible $\sigma$-families, $F_1, F_2, F_3$ one for each of the standard $\T$-invariant open subsets $U_1,U_2,U_3$ of $\P^2$. For any element $m$ of the character group of $\T$ identified with $\Z^2$, $F_i(m)=\Gamma(U_i,\F)_m$, the eigenspace corresponding to $m$ in the space of sections of $\F$ on  $U_i$.
A triple of $\sigma$-families giving rise to a $\T$-equivariant sheaf is called a $\Delta$-family.

A $\T$-equivariant rank 1 torsion-free sheaf $\I$ on $\P^2$ is determined by three integers $u,v,w$ and three 2d partitions $\pi'_1, \pi'_2, \pi'_3.$ Figure \ref{fig:rank1} from left to right indicates the $\sigma$-families $I_1, I_2, I_3$ over $U_1,U_2,U_3$, respectively.
\\
\begin{figure}[h]

\ifx\JPicScale\undefined\def\JPicScale{1}\fi
\unitlength \JPicScale mm
\setlength{\unitlength}{0.032in}

\begin{picture}(135,30)(0,60)
\linethickness{0.3mm}
\put(10,60){\line(0,1){30}}
\linethickness{0.3mm}
\put(5,65){\line(1,0){35}}
\put(3,60){\makebox(0,0)[cc]{$(u,v)$}}

\put(15,78){\makebox(0,0)[cc]{$\pi'_{1}$}}

\put(70,78){\makebox(0,0)[cc]{$\pi'_{2}$}}
\put(115,78){\makebox(0,0)[cc]{$\pi'_{3}$}}

\put(10,65){\makebox(0,0)[cc]{$\bullet$}}

\linethickness{0.3mm}
\put(60,60){\line(0,1){30}}
\linethickness{0.3mm}
\put(55,65){\line(1,0){35}}
\put(52,60){\makebox(0,0)[cc]{$(v,w)$}}

\put(60,65){\makebox(0,0)[cc]{$\bullet$}}

\linethickness{0.3mm}
\put(105,60){\line(0,1){30}}
\linethickness{0.3mm}
\put(100,65){\line(1,0){35}}
\put(98,60){\makebox(0,0)[cc]{$(w,u)$}}

\put(105,65){\makebox(0,0)[cc]{$\bullet$}}

\linethickness{0.3mm}
\put(10,75){\line(1,0){5}}
\linethickness{0.3mm}
\put(15,70){\line(0,1){5}}
\linethickness{0.3mm}
\put(15,70){\line(1,0){5}}
\linethickness{0.3mm}
\put(20,65){\line(0,1){5}}
\linethickness{0.3mm}
\put(60,85){\line(1,0){5}}
\linethickness{0.3mm}
\put(65,70){\line(0,1){15}}
\linethickness{0.3mm}

\put(65,70){\line(1,0){10}}
\linethickness{0.3mm}
\put(75,65){\line(0,1){5}}
\linethickness{0.3mm}
\put(105,75){\line(1,0){5}}
\linethickness{0.3mm}
\put(110,70){\line(0,1){5}}
\linethickness{0.3mm}
\put(110,70){\line(1,0){15}}
\linethickness{0.3mm}
\put(125,65){\line(0,1){5}}
\end{picture}
\caption{$\Delta$-family for a $\T$-equivariant rank 1 sheaf}
\label{fig:rank1}
\end{figure}
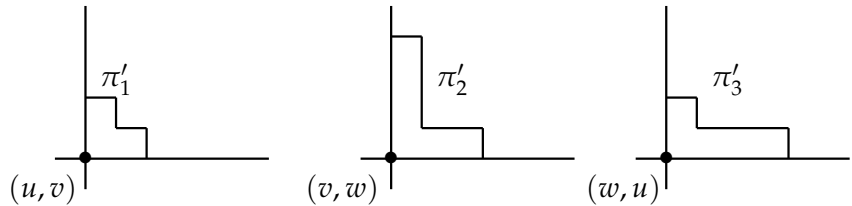

\vspace{4cm}
For any $j=1,2,3$ we have $I_j(m)=0$ if $m$ is below the horizontal axis, on the left of the vertical axis, or inside the partition $\pi'_j$. Otherwise, $I_j(m)=\CC$. 
Moreover, we have (see \cite[Proposition 3.16]{a34}) \begin{align*}\ch_1(\I)&=-u-v-w, \\\ch_2(\I)&=(u+v+w)^2/2-\#\pi'_1-\#\pi'_2-\#\pi'_3. \end{align*}

A $\T$-equivariant rank 2 torsion-free sheaf $\F$ on $\P^2$ (up to tensoring with a degree 0 $\T$-equivariant line bundle) is determined by a $\Delta$-family given by 
\begin{enumerate}[i)]
\item an integer $A$,
\item three nonnegative integers $\Delta_1, \Delta_2, \Delta_3$,
\item a 1-dimensional subspace $p_i \subset \CC^2$ one for each $\Delta_i\neq 0$ (if $\Delta_i=0$ we set $p_i=0$),
\item six 2d partitions $\pi_j^1, \pi_j^2$ for $j=1,2,3$.  
\end{enumerate}

\begin{defn}
We call a $\Delta$-family corresponding to a $\T$-equivariant rank 2 torsion-free sheaf $\F$ on $\P^2$ \emph{non-degenerate} if $p_1,p_2,p_3$ are pairwise distinct and nonzero.
\hfill $\oslash$
\end{defn}
Figure \ref{fig:rank2}  indicates the corresponding $\Delta$-family $(F_1, F_2, F_3)$  in the case $p_1,p_2,p_3$ are pairwise distinct  \footnote{each diagram corresponds to the $\sigma$-family $F_i$ on $U_i$.}. The points indicated by $\bullet$ have the coordinates $(0,0),\; (0,A),\; (A,0)$, respectively. The partitions $\pi_j^1, \pi_j^2$ are placed  respectively at the points (indicated by $\times$) with the coordinates $$(0,\Delta_2), (\Delta_1,0), (0,A+\Delta_3), (\Delta_2,A), (A,\Delta_1), (A+\Delta_3,0).$$

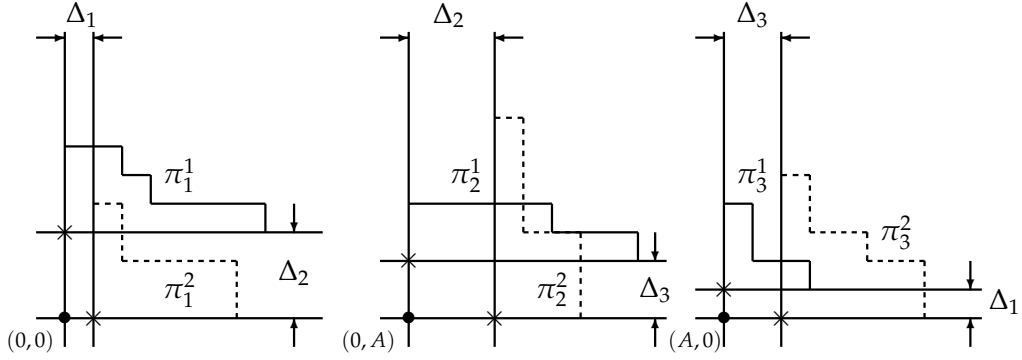
\begin{figure}
\ifx\JPicScale\undefined\def\JPicScale{1}\fi
\unitlength \JPicScale mm
\setlength{\unitlength}{0.03in}

\begin{picture}(165,80)(0,35)
\linethickness{0.3mm}
\put(5,35){\line(0,1){55}}
\linethickness{0.3mm}

\linethickness{0.3mm}
\put(10,89){\line(1,0){5}}
\put(10,89){\vector(-1,0){0.12}}
\linethickness{0.3mm}
\put(0,89){\line(1,0){5}}
\put(5,89){\vector(1,0){0.12}}

\linethickness{0.3mm}
\put(80,89){\line(1,0){5}}
\put(80,89){\vector(-1,0){0.12}}
\linethickness{0.3mm}
\put(60,89){\line(1,0){5}}
\put(65,89){\vector(1,0){0.12}}

\linethickness{0.3mm}
\put(130,89){\line(1,0){5}}
\put(130,89){\vector(-1,0){0.12}}
\linethickness{0.3mm}
\put(115,89){\line(1,0){5}}
\put(120,89){\vector(1,0){0.12}}

\linethickness{0.3mm}
\put(45,55){\line(0,1){5}}
\put(45,55){\vector(0,-1){0.12}}
\linethickness{0.3mm}
\put(45,35){\line(0,1){5}}
\put(45,40){\vector(0,1){0.12}}

\linethickness{0.3mm}
\put(108,50){\line(0,1){5}}
\put(108,50){\vector(0,-1){0.12}}
\linethickness{0.3mm}
\put(108,35){\line(0,1){5}}
\put(108,40){\vector(0,1){0.12}}

\linethickness{0.3mm}
\put(163,45){\line(0,1){5}}
\put(163,45){\vector(0,-1){0.12}}
\linethickness{0.3mm}
\put(163,35){\line(0,1){5}}
\put(163,40){\vector(0,1){0.12}}

\put(0,40){\line(1,0){50}}
\linethickness{0.3mm}
\put(10,35){\line(0,1){55}}
\linethickness{0.3mm}
\put(0,55){\line(1,0){50}}
\linethickness{0.3mm}
\put(65,35){\line(0,1){55}}
\linethickness{0.3mm}
\put(60,40){\line(1,0){50}}
\linethickness{0.3mm}
\put(120,35){\line(0,1){55}}
\linethickness{0.3mm}
\put(115,40){\line(1,0){50}}
\linethickness{0.3mm}
\put(60,50){\line(1,0){50}}
\linethickness{0.3mm}
\put(80,35){\line(0,1){55}}
\linethickness{0.3mm}
\put(130,35){\line(0,1){55}}
\linethickness{0.3mm}
\put(115,45){\line(1,0){50}}
\put(8,93){\makebox(0,0)[cc]{$\Delta_{1}$}}

\put(45,48){\makebox(0,0)[cc]{$\Delta_{2}$}}

\put(25,65){\makebox(0,0)[cc]{$\pi_{1}^1$}}

\put(25,45){\makebox(0,0)[cc]{$\pi_{1}^2$}}

\put(75,65){\makebox(0,0)[cc]{$\pi_{2}^1$}}

\put(90,45){\makebox(0,0)[cc]{$\pi_{2}^2$}}

\put(125,65){\makebox(0,0)[cc]{$\pi_{3}^1$}}

\put(150,55){\makebox(0,0)[cc]{$\pi_{3}^2$}}

\put(72,93){\makebox(0,0)[cc]{$\Delta_{2}$}}

\put(108,45){\makebox(0,0)[cc]{$\Delta_{3}$}}

\put(125,93){\makebox(0,0)[cc]{$\Delta_{3}$}}

\put(169,43){\makebox(0,0)[cc]{$\Delta_{1}$}}

\linethickness{0.3mm}
\put(5,70){\line(1,0){10}}
\linethickness{0.3mm}
\put(15,65){\line(0,1){5}}
\linethickness{0.3mm}
\put(15,65){\line(1,0){5}}
\linethickness{0.3mm}
\put(20,60){\line(0,1){5}}
\linethickness{0.3mm}
\put(20,60){\line(1,0){20}}
\linethickness{0.3mm}
\put(40,55){\line(0,1){5}}
\linethickness{0.3mm}
\multiput(10,60)(2,0){3}{\line(1,0){1}}
\linethickness{0.3mm}
\multiput(15,50)(1.9,0){11}{\line(1,0){0.95}}
\linethickness{0.3mm}
\multiput(35,40)(0,1.82){6}{\line(0,1){0.91}}
\linethickness{0.3mm}
\multiput(15,50)(0,1.82){6}{\line(0,1){0.91}}
\linethickness{0.3mm}
\multiput(95,40)(0,2){8}{\line(0,1){1}}
\linethickness{0.3mm}
\multiput(85,55)(1.82,0){6}{\line(1,0){0.91}}
\linethickness{0.3mm}
\multiput(85,55)(0,1.9){11}{\line(0,1){0.95}}
\linethickness{0.3mm}
\multiput(80,75)(2,0){3}{\line(1,0){1}}
\linethickness{0.3mm}
\put(65,60){\line(1,0){25}}
\linethickness{0.3mm}
\put(90,55){\line(0,1){5}}
\linethickness{0.3mm}
\put(90,55){\line(1,0){15}}
\linethickness{0.3mm}
\put(105,50){\line(0,1){5}}
\linethickness{0.3mm}
\put(120,60){\line(1,0){5}}
\linethickness{0.3mm}
\put(125,50){\line(0,1){10}}
\linethickness{0.3mm}
\put(125,50){\line(1,0){10}}
\linethickness{0.3mm}
\put(135,45){\line(0,1){5}}
\linethickness{0.3mm}
\multiput(155,40)(0,1.82){6}{\line(0,1){0.91}}
\linethickness{0.3mm}
\multiput(145,50)(1.82,0){6}{\line(1,0){0.91}}
\linethickness{0.3mm}
\multiput(145,50)(0,2){3}{\line(0,1){1}}
\linethickness{0.3mm}
\multiput(135,55)(1.82,0){6}{\line(1,0){0.91}}
\linethickness{0.3mm}
\multiput(135,55)(0,1.82){6}{\line(0,1){0.91}}
\linethickness{0.3mm}
\multiput(130,65)(2,0){3}{\line(1,0){1}}


\put(5,40){\makebox(0,0)[cc]{$\bullet$}}
\put(65,40){\makebox(0,0)[cc]{$\bullet$}}
\put(120,40){\makebox(0,0)[cc]{$\bullet$}}

\put(5,55){\makebox(0,0)[cc]{$\times$}}

\put(10,40){\makebox(0,0)[cc]{$\times$}}

\put(65,50){\makebox(0,0)[cc]{$\times $}}

\put(120,45){\makebox(0,0)[cc]{$\times$}}

\put(130,40){\makebox(0,0)[cc]{$\times $}}

\put(80,40){\makebox(0,0)[cc]{$\times$}}

\fontsize{8pt}{12pt}\selectfont{

\put(-1,36){\makebox(0,0)[cc]{$(0,0)$}}

\put(58,36){\makebox(0,0)[cc]{$(0,A)$}}

\put(115,36){\makebox(0,0)[cc]{$(A,0)$}}}

\end{picture}
\caption{A non-degenerate $\Delta$-family corresponding to a rank 2 $\T$-equivariant sheaf}
\label{fig:rank2}
\end{figure}

Figure \ref{fig:rank02}, indicates a typical degenerate case in which $p_1=p_2\neq p_3$. Everything is the same as in the non-degenerate case except that we reposition the partitions $\pi_1^1, \pi_1^2$ to respectively the points with the coordinates $(0,0), (\Delta_1,\Delta_2).$ Similarly if $p_1\neq p_2=p_3$ (respectively  $p_1=p_2\neq p_3$) then we reposition the partitions $\pi_2^1, \pi_2^2$  (respectively $\pi_3^1, \pi_3^2$) to the points with the coordinates $(0,A), (\Delta_2,A+\Delta_3)$ (respectively $(A,0), (A+\Delta_3,\Delta_1)$). Finally, if $p_1=p_2=p_3$ then we bring all the partitions to the point with new coordinates (indicated by $\bullet$).

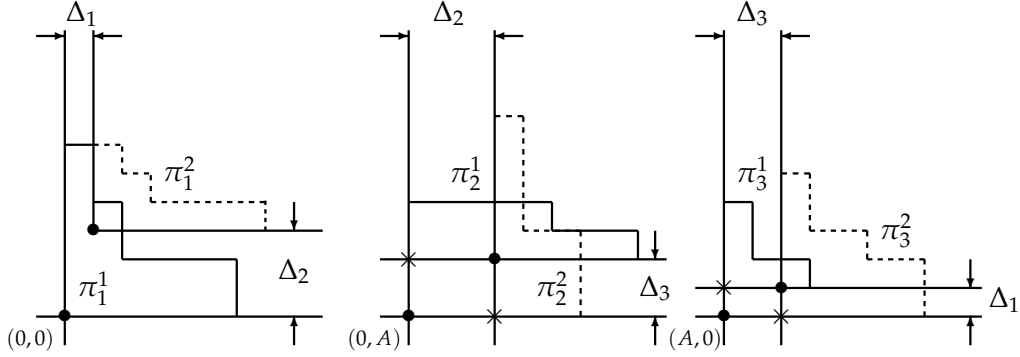
\begin{figure}
\ifx\JPicScale\undefined\def\JPicScale{1}\fi
\unitlength \JPicScale mm
\setlength{\unitlength}{0.03in}

\begin{picture}(165,75)(0,35)
\linethickness{0.3mm}
\put(5,35){\line(0,1){55}}
\linethickness{0.3mm}

\linethickness{0.3mm}
\put(10,89){\line(1,0){5}}
\put(10,89){\vector(-1,0){0.12}}
\linethickness{0.3mm}
\put(0,89){\line(1,0){5}}
\put(5,89){\vector(1,0){0.12}}

\linethickness{0.3mm}
\put(80,89){\line(1,0){5}}
\put(80,89){\vector(-1,0){0.12}}
\linethickness{0.3mm}
\put(60,89){\line(1,0){5}}
\put(65,89){\vector(1,0){0.12}}

\linethickness{0.3mm}
\put(130,89){\line(1,0){5}}
\put(130,89){\vector(-1,0){0.12}}
\linethickness{0.3mm}
\put(115,89){\line(1,0){5}}
\put(120,89){\vector(1,0){0.12}}

\linethickness{0.3mm}
\put(45,55){\line(0,1){5}}
\put(45,55){\vector(0,-1){0.12}}
\linethickness{0.3mm}
\put(45,35){\line(0,1){5}}
\put(45,40){\vector(0,1){0.12}}

\linethickness{0.3mm}
\put(108,50){\line(0,1){5}}
\put(108,50){\vector(0,-1){0.12}}
\linethickness{0.3mm}
\put(108,35){\line(0,1){5}}
\put(108,40){\vector(0,1){0.12}}

\linethickness{0.3mm}
\put(163,45){\line(0,1){5}}
\put(163,45){\vector(0,-1){0.12}}
\linethickness{0.3mm}
\put(163,35){\line(0,1){5}}
\put(163,40){\vector(0,1){0.12}}

\put(0,40){\line(1,0){50}}
\linethickness{0.3mm}
\put(10,55){\line(0,1){35}}
\linethickness{0.3mm}
\put(10,55){\line(1,0){40}}
\linethickness{0.3mm}
\put(65,35){\line(0,1){55}}
\linethickness{0.3mm}
\put(60,40){\line(1,0){50}}
\linethickness{0.3mm}
\put(120,35){\line(0,1){55}}
\linethickness{0.3mm}
\put(115,40){\line(1,0){50}}
\linethickness{0.3mm}
\put(60,50){\line(1,0){50}}
\linethickness{0.3mm}
\put(80,35){\line(0,1){55}}
\linethickness{0.3mm}
\put(130,35){\line(0,1){55}}
\linethickness{0.3mm}
\put(115,45){\line(1,0){50}}
\put(8,93){\makebox(0,0)[cc]{$\Delta_{1}$}}

\put(45,48){\makebox(0,0)[cc]{$\Delta_{2}$}}

\put(25,65){\makebox(0,0)[cc]{$\pi_{1}^2$}}

\put(10,45){\makebox(0,0)[cc]{$\pi_{1}^1$}}

\put(75,65){\makebox(0,0)[cc]{$\pi_{2}^1$}}

\put(90,45){\makebox(0,0)[cc]{$\pi_{2}^2$}}

\put(125,65){\makebox(0,0)[cc]{$\pi_{3}^1$}}

\put(150,55){\makebox(0,0)[cc]{$\pi_{3}^2$}}

\put(72,93){\makebox(0,0)[cc]{$\Delta_{2}$}}

\put(108,45){\makebox(0,0)[cc]{$\Delta_{3}$}}

\put(125,93){\makebox(0,0)[cc]{$\Delta_{3}$}}

\put(169,43){\makebox(0,0)[cc]{$\Delta_{1}$}}

\linethickness{0.3mm}
\put(5,70){\line(1,0){5}}
\multiput(10,70)(1.9,0){3}{\line(1,0){.9}}

\linethickness{0.3mm}
\multiput(15,65)(0,1.9){3}{\line(0,1){.9}}
\linethickness{0.3mm}
\multiput(15,65)(1.9,0){3}{\line(1,0){.9}}
\linethickness{0.3mm}
\multiput(20,60)(0,1.9){3}{\line(0,1){.9}}
\linethickness{0.3mm}
\multiput(20,60)(1.9,0){11}{\line(1,0){.95}}
\linethickness{0.3mm}
\multiput(40,55)(0,1.5){4}{\line(0,1){.75}}
\linethickness{0.3mm}
\put(10,60){\line(1,0){5}}
\linethickness{0.3mm}
\put(15,50){\line(1,0){20}}
\linethickness{0.3mm}
\put(35,40){\line(0,1){10}}
\linethickness{0.3mm}
\put(15,50){\line(0,1){10}}

\linethickness{0.3mm}
\multiput(95,40)(0,2){8}{\line(0,1){1}}
\linethickness{0.3mm}
\multiput(85,55)(1.82,0){6}{\line(1,0){0.91}}
\linethickness{0.3mm}
\multiput(85,55)(0,1.9){11}{\line(0,1){0.95}}
\linethickness{0.3mm}
\multiput(80,75)(2,0){3}{\line(1,0){1}}
\linethickness{0.3mm}
\put(65,60){\line(1,0){25}}
\linethickness{0.3mm}
\put(90,55){\line(0,1){5}}
\linethickness{0.3mm}
\put(90,55){\line(1,0){15}}
\linethickness{0.3mm}
\put(105,50){\line(0,1){5}}
\linethickness{0.3mm}
\put(120,60){\line(1,0){5}}
\linethickness{0.3mm}
\put(125,50){\line(0,1){10}}
\linethickness{0.3mm}
\put(125,50){\line(1,0){10}}
\linethickness{0.3mm}
\put(135,45){\line(0,1){5}}
\linethickness{0.3mm}
\multiput(155,40)(0,1.82){6}{\line(0,1){0.91}}
\linethickness{0.3mm}
\multiput(145,50)(1.82,0){6}{\line(1,0){0.91}}
\linethickness{0.3mm}
\multiput(145,50)(0,2){3}{\line(0,1){1}}
\linethickness{0.3mm}
\multiput(135,55)(1.82,0){6}{\line(1,0){0.91}}
\linethickness{0.3mm}
\multiput(135,55)(0,1.82){6}{\line(0,1){0.91}}
\linethickness{0.3mm}
\multiput(130,65)(2,0){3}{\line(1,0){1}}

\put(10,55){\makebox(0,0)[cc]{$\bullet$}}
\put(80,50){\makebox(0,0)[cc]{$\bullet$}}

\put(5,40){\makebox(0,0)[cc]{$\bullet$}}
\put(65,40){\makebox(0,0)[cc]{$\bullet$}}
\put(120,40){\makebox(0,0)[cc]{$\bullet$}}
\put(130,45){\makebox(0,0)[cc]{$\bullet$}}



\put(65,50){\makebox(0,0)[cc]{$\times$}}

\put(120,45){\makebox(0,0)[cc]{$\times$}}

\put(80,40){\makebox(0,0)[cc]{$\times$}}

\put(130,40){\makebox(0,0)[cc]{$\times$}}

\fontsize{8pt}{12pt}\selectfont{
\put(-1,36){\makebox(0,0)[cc]{$(0,0)$}}

\put(59,36){\makebox(0,0)[cc]{$(0,A)$}}

\put(115,36){\makebox(0,0)[cc]{$(A,0)$}}}

\end{picture}

\caption{A degenerate $\Delta$-family corresponding to a rank 2 $\T$-equivariant sheaf}
\label{fig:rank02}
\end{figure}

For any $j=1,2,3$, we denote by $S_j^1,S_j^2$ the vertical and the horizontal strips made by two vertical and two horizontal lines in each diagram:
$$S_1^1=\{(x,y)| x\ge0,\; 0\le y < \Delta_1\}, \quad S_1^2=\{(x,y)| y\ge0,\; 0\le x < \Delta_2\}$$  
etc. We also denote by $R_1,R_2,R_3$ the areas located above the horizontal strip and to the right of the vertical strip:
$$R_1=\{(x,y)| x\ge \Delta_1,\;  y \ge \Delta_2\}$$ etc. We then have (we use the convention $p_4=p_1$)
\begin{enumerate}[(F1)]
\item $F_j(m)=0$ if either $m\in \pi_j^1\cap \pi_j^2$,  $\;m\in S_j^i\cap \pi_j^{i'}$,  $\;m$ is on the left of the strip $S_j^1$, or $m$ is below the strip $S_j^2$.
\item $F_j(m)=p_j\cap p_{j+1}$,  $m\in S_i^1\cap S_i^2- \pi_j^1$. 
\item $F_j(m)=\CC^2$ if $m$ is in $R_j-\pi_j^1-\pi_j^2$.
\item Suppose that $p_j\neq p_{j+1}$ then $F_j(m)=p_j$  if $m\in S^1_j-\pi_j^1$ or if $m$ belongs to a connected component of  $\pi_j^2-\pi_j^1$ adjacent to a member of $S^1_j-\pi_j^1$; and  $F_j(m)=p_{j+1}$  if $m\in S^2_j-\pi_j^2$ or if $m$ belongs to a connected component of  $\pi_j^1-\pi_j^2$ adjacent to a member of $m\in S^2_j-\pi_j^2$.

\noindent Suppose that $p_j=p_{j+1}$ then $F_j(m)=p_j$  if $m\in S^1_j\cup S^2_j-\pi_j^1$ or if $m$ belongs to a connected component of  $\pi_j^1\cup \pi_j^2-\pi_j^1\cap \pi_j^2$ adjacent to a member of $S^1_j\cup S^2_j-\pi_j^1$.  
\item $F_j(m)=s \subset \CC^2$ where $s$ is an arbitrary 1-dimensional subspace of $\CC^2$ for all $m$ in any connected component of $\pi_j^1\cup\pi_j^2-\pi_j^1\cap \pi_j^2$ other than the ones mentioned in (F4). We denote these connected components by $C_1,\dots,C_k$, and denote by $s_1,\dots,s_k\subset \CC^2$ the corresponding 1-dimensional subspaces.
\end{enumerate}
\begin{rem}
It can be seen that the $\T$-equivariant sheaf described above is indeomposable if and only if the number of nonzero elements of the set $\{p_1,p_2,p_3, s_1,\dots,s_k\}$ is at least three. A $\T$-equivariant sheaf with non-degenerate $\Delta$-family obviously satisfies this condition regardless of the subspaces $s_1,\dots,s_k$. 

\hfill $\oslash$
\end{rem}

\begin{defn}
By a \emph{$\Delta$-family data} $\Xi$ we mean the collection of 
\begin{enumerate} [i)]
\item an integer $A$,
\item nonnegative integers $\Delta_1,\Delta_2, \Delta_3$,
\item six possibly empty 2d partitions $\pi_1^1,\dots,\pi_3^2$,
\item the set $E=\{(i,j)| \; i<j,\; p_i\neq0,\; p_j\neq 0, \; p_i= p_j\}$.
\end{enumerate}
Note that in a $\Delta$-family data we do not specify $p_1, p_2, p_3$, but we only keep track of whether they are distinct or not. Given a $\Delta$-family data $\Xi$ we define $e_{\Xi}$ to be the number of nonzero elements of the set $\{p_1,p_2,p_3\}$.

\hfill $\oslash$
\end{defn}

\newcommand{\sD}{\mathsf{D}}
In terms of $\Delta$-family data (see \cite[Proposition 3.16]{a34}) \begin{align*}\ch_1(\F)&=-2A-\Delta_1-\Delta_2-\Delta_3, \\ \ch_2(\F)&=A^2/2+(A+\sum_{i}\Delta_i)^2/2-\sum_{i,j} \#\pi^i_j-\sum_{i<j}\Delta_i\Delta_j(1-\dim p_i\cap p_j).\end{align*} 
As a result, the $\Delta$-family of the $\T$-equivariant sheaf $\F$ determines the Hilbert polynomial of $\F$.  
\begin{defn}For a Hilbert polynomial $P$ of a rank 2 torsion-free sheaf $\F$ on $\P^2$, we define $\mathsf{D}(P)$ to be the set of all $\Delta$-family data giving rise to $P$.
\hfill $\oslash$
\end{defn}
Given a non-degenerate $\F$ as above, we define the rank 1 torsion-free $\T$-equivariant sheaves $\L_{p_i}$ for any nonzero $p_i$, to be the maximal subsheaves of $\F$ respectively with 
\begin{align*}
&
u=0,v=\Delta_2, w=A+\Delta_3,\;\; \text{generated by }\;\; p_1 \notag\\
&
u=\Delta_1,v=0, w=A+\Delta_3,\;\; \text{generated by }\;\; p_2\notag\\
&
u=\Delta_1,v=\Delta_2, w=A,\;\; \text{generated by }\;\; p_3 
\end{align*}
Similarly, if $\F$ is given by a degenerate $\sigma$-family we can define $\L_{p_i}$ to be the maximal rank 1 subsheaf generated by the $p_i$.

We are only interested in the case where the Hilbert polynomial of $\F$ is $P(m)=m^2+3m+2+b$, so we must have 
$$\Delta_1+\Delta_2+\Delta_3=-2A,$$ and $$b=A^2-\sum_{i,j} \#\pi^i_j-\sum_{i<j}\Delta_i\Delta_j(1-\dim p_i\cap p_j).$$ Then one can see that $\F$ is Gieseker (semi)stable if the Hilbert polynomial of $\L_{p_i}$ is less than (less than or equal to) $P/2$ for $i=1,2,3$ (see \cite[Proposition 3.19]{a34}). Similarly $\F$ is $\mu$-(semi)stable if the linear term of the Hilbert polynomial of $\L_{p_i}$ is less than (less than or equal to) to the linear term of $P/2$ for $i=1,2,3$.

A closed point of $\cP_n(\P^2,P)^{\T}$ consists of a $\T$-equivariant semistable sheaf $\F$ and a $\T$-invariant morphism $\O(-n)\xrightarrow{f} \F$, such that $f$ does not factor through any $\L_{p_i}$ with the Hilbert polynomial $P/2$. Given an indecomposable $\T$-equivariant semistable sheaf $\F$, such an $f$ is determined by a triple of integers $\ell=(u,v,w)$ and a 1-dimensional subspace $t\subset \CC^2$ such that
\begin{enumerate} [(P1)]
\item $u+v+w=n$.
\item $\ell$ determines 3 lattice points $(u,v)$, $(v,w)$, and $(w,u)$ one in each diagram of the $\Delta$-family. We require that $F_1(u,v)\neq 0$, $F_2(v,w)\neq 0$, $F_3(w,u)\neq 0$.
\item If one of $F_1(u,v)$, $F_2(v,w)$, or $F_3(w,u)$ is a 1-dimensional subspace $r\subset \CC^2$ then $t$ has to be equal to $r$, 
\item If the Hilbert polynomial of $\L_{p_i}$ is $P/2$, then $t$ is not allowed to be equal to $p_i$.
\end{enumerate}
\begin{defn}We define $l(\Xi,\ell)$ to be 0 if (P3) is satisfied, otherwise we set $l(\Xi,\ell)=1$. We sometimes use $l$ instead of $l(\Xi,\ell)$ if $\Xi$ and $\ell$ are clear from the context. Given a $\Delta$-family data $\Xi$ and a triple of integers $\ell$ satisfying (P1)-(P4) we say that $\ell$ is \emph{$\Xi$-compatible}. 
\hfill $\oslash$
\end{defn}

For a fixed $\Xi$ and $\ell$ (compatible with $\Xi$), any closed point of $C(\Xi,\ell)=(\P^1)^{k+e_\Xi+l}$ gives rise to a $\T$-equivariant pair $$\O(-n)\xrightarrow{f} \F$$ where $\F$ is a rank 2 $\T$-equivariant sheaf with $\Delta$-family data $\Xi$ \footnote{Note that 1-dimensional subspaces of $\CC^2$ can be identified with the closed points of $\text{Gr(1,2)}\cong\P^1.$}.

\begin{example} \label{ex:toric}
In this example we consider a $\T$-equivariant stable pair $$\O(-5)\xrightarrow{f} \F$$ with the non-degenerate $\Delta$-family data $\Xi$ given by $$A=-2, \quad (\Delta_1,\Delta_2,\Delta_3)=(2,1,1),\quad (\pi^1_1,\dots,\pi_3^2)=(\s,\;,\;,\;,\;,\;).$$ The Hilbert polynomial of $\F$ is $P(m)=m^2+3m$. The lattice points $$(u,v), (v,w), (w,u)$$ corresponding to the possible $\Xi$-compatible triples $\ell=(u,v,w)$ are denoted by $\circ$ and $\bullet$ in Figure \ref{fig:toric} respectively from left to right. We use $\circ$ when $l(\Xi,\ell)=0$, and $\bullet$ when $l(\Xi,\ell)=1$. Note that in each diagram the number of $\circ$ plus twice the number of $\bullet$ is equal to $P(5)=40$. 

In order to determine the morphism $f$, we in addition need to specify a 1-dimensional subspace $t\subset \CC^2$.
Note that $\F$ given by this data is strictly semistable if the Hilbert polynomial of $\L_{p_1}$ is equal to $P/2$. Therefore, the pair is stable if and only if $t\neq p_1$. This means that if $l(\Xi,\ell)=0$ then there is at most one choice for $t$; but if $l(\Xi,\ell)=1$ then $t$ can be any point in $\P^1 \backslash p_1$. So the space of possible $t$ can be either $\emptyset$, a point, or $\P^1 \backslash p_1$. We assign the weight 0 or 1 to a lattice point $\circ$ depending on respectively the first or second possibility occurs, and we assign the weight 1 to each $\bullet$. With these new weights, one can check that the count of $\circ$ and $\bullet$ in each diagram is P(5)/2=20. 
\end{example}
\begin{figure}[h] \label{fig:toric}
\ifx\JPicScale\undefined\def\JPicScale{1}\fi
\unitlength \JPicScale mm
\begin{picture}(135,45)(10,50)
\linethickness{0.3mm}
\put(10,50){\line(0,1){40}}
\linethickness{0.3mm}
\put(5,55){\line(1,0){40}}
\linethickness{0.3mm}
\put(20,50){\line(0,1){40}}
\linethickness{0.3mm}
\put(5,60){\line(1,0){40}}
\linethickness{0.3mm}
\put(55,50){\line(0,1){40}}
\linethickness{0.3mm}
\put(50,55){\line(1,0){40}}
\linethickness{0.3mm}
\put(60,50){\line(0,1){40}}
\linethickness{0.3mm}
\put(50,60){\line(1,0){40}}
\linethickness{0.3mm}
\put(100,50){\line(0,1){40}}
\linethickness{0.3mm}
\put(95,55){\line(1,0){40}}
\linethickness{0.3mm}
\put(105,50){\line(0,1){40}}
\linethickness{0.3mm}
\put(95,65){\line(1,0){40}}
\linethickness{0.3mm}
\put(5,90){\line(1,0){5}}
\put(10,90){\vector(1,0){0.12}}
\linethickness{0.3mm}
\put(20,90){\line(1,0){5}}
\put(20,90){\vector(-1,0){0.12}}
\linethickness{0.3mm}
\put(50,90){\line(1,0){5}}
\put(55,90){\vector(1,0){0.12}}
\linethickness{0.3mm}
\put(60,90){\line(1,0){5}}
\put(60,90){\vector(-1,0){0.12}}
\linethickness{0.3mm}
\put(95,90){\line(1,0){5}}
\put(100,90){\vector(1,0){0.12}}
\linethickness{0.3mm}
\put(105,90){\line(1,0){5}}
\put(105,90){\vector(-1,0){0.12}}
\linethickness{0.3mm}
\put(45,60){\line(0,1){5}}
\put(45,60){\vector(0,-1){0.12}}
\linethickness{0.3mm}
\put(45,50){\line(0,1){5}}
\put(45,55){\vector(0,1){0.12}}
\linethickness{0.3mm}
\put(90,60){\line(0,1){5}}
\put(90,60){\vector(0,-1){0.12}}
\linethickness{0.3mm}
\put(90,50){\line(0,1){5}}
\put(90,55){\vector(0,1){0.12}}
\linethickness{0.3mm}
\put(135,65){\line(0,1){5}}
\put(135,65){\vector(0,-1){0.12}}
\linethickness{0.3mm}
\put(135,50){\line(0,1){5}}
\put(135,55){\vector(0,1){0.12}}
\put(15,95){\makebox(0,0)[cc]{$\Delta_{1}=2$}}
\put(45,67){\makebox(0,0)[cc]{$\Delta_{2}=1$}}
\put(60,95){\makebox(0,0)[cc]{$\Delta_{2}=1$}}
\put(90,67){\makebox(0,0)[cc]{$\Delta_{3}=1$}}
\put(100,95){\makebox(0,0)[cc]{$\Delta_{3}=1$}}
\put(135,60){\makebox(0,0)[cc]{$\Delta_{1}=2$}}
\linethickness{0.3mm}
\put(10,65){\line(1,0){5}}
\linethickness{0.3mm}
\put(15,60){\line(0,1){5}}
\put(10,65){\makebox(0,0)[cc]{$\circ$}}
\put(10,70){\makebox(0,0)[cc]{$\circ$}}
\put(10,75){\makebox(0,0)[cc]{$\circ$}}
\put(10,80){\makebox(0,0)[cc]{$\circ$}}
\put(10,85){\makebox(0,0)[cc]{$\circ$}}
\put(15,65){\makebox(0,0)[cc]{$\circ$}}
\put(15,70){\makebox(0,0)[cc]{$\circ$}}
\put(15,75){\makebox(0,0)[cc]{$\circ$}}
\put(15,80){\makebox(0,0)[cc]{$\circ$}}
\put(20,60){\makebox(0,0)[cc]{$\bullet$}}
\put(20,65){\makebox(0,0)[cc]{$\bullet$}}
\put(20,70){\makebox(0,0)[cc]{$\bullet$}}
\put(20,75){\makebox(0,0)[cc]{$\bullet$}}
\put(20,80){\makebox(0,0)[cc]{$\circ$}}
\put(25,60){\makebox(0,0)[cc]{$\bullet$}}
\put(25,65){\makebox(0,0)[cc]{$\bullet$}}
\put(25,70){\makebox(0,0)[cc]{$\bullet$}}
\put(25,75){\makebox(0,0)[cc]{$\circ$}}
\put(30,60){\makebox(0,0)[cc]{$\bullet$}}
\put(30,65){\makebox(0,0)[cc]{$\bullet$}}
\put(30,70){\makebox(0,0)[cc]{$\circ$}}
\put(35,60){\makebox(0,0)[cc]{$\bullet$}}
\put(35,65){\makebox(0,0)[cc]{$\circ$}}
\put(40,60){\makebox(0,0)[cc]{$\circ$}}
\put(20,55){\makebox(0,0)[cc]{$\circ$}}
\put(25,55){\makebox(0,0)[cc]{$\circ$}}
\put(30,55){\makebox(0,0)[cc]{$\circ$}}
\put(35,55){\makebox(0,0)[cc]{$\circ$}}
\put(40,55){\makebox(0,0)[cc]{$\circ$}}
\put(15,60){\makebox(0,0)[cc]{$\circ$}}
\put(55,60){\makebox(0,0)[cc]{$\circ$}}
\put(55,65){\makebox(0,0)[cc]{$\circ$}}
\put(55,70){\makebox(0,0)[cc]{$\circ$}}
\put(55,75){\makebox(0,0)[cc]{$\circ$}}
\put(55,80){\makebox(0,0)[cc]{$\circ$}}
\put(60,80){\makebox(0,0)[cc]{$\circ$}}
\put(65,80){\makebox(0,0)[cc]{$\circ$}}
\put(65,75){\makebox(0,0)[cc]{$\circ$}}
\put(70,75){\makebox(0,0)[cc]{$\circ$}}
\put(70,70){\makebox(0,0)[cc]{$\circ$}}
\put(75,70){\makebox(0,0)[cc]{$\circ$}}
\put(75,65){\makebox(0,0)[cc]{$\circ$}}
\put(80,65){\makebox(0,0)[cc]{$\circ$}}
\put(80,60){\makebox(0,0)[cc]{$\circ$}}
\put(85,60){\makebox(0,0)[cc]{$\circ$}}
\put(85,55){\makebox(0,0)[cc]{$\circ$}}
\put(65,55){\makebox(0,0)[cc]{$\circ$}}
\put(70,55){\makebox(0,0)[cc]{$\circ$}}
\put(75,55){\makebox(0,0)[cc]{$\circ$}}
\put(80,55){\makebox(0,0)[cc]{$\circ$}}
\put(60,60){\makebox(0,0)[cc]{$\bullet$}}
\put(60,65){\makebox(0,0)[cc]{$\bullet$}}
\put(60,70){\makebox(0,0)[cc]{$\bullet$}}
\put(60,75){\makebox(0,0)[cc]{$\bullet$}}
\put(65,60){\makebox(0,0)[cc]{$\bullet$}}
\put(65,65){\makebox(0,0)[cc]{$\bullet$}}
\put(65,70){\makebox(0,0)[cc]{$\bullet$}}
\put(70,60){\makebox(0,0)[cc]{$\bullet$}}
\put(70,65){\makebox(0,0)[cc]{$\bullet$}}
\put(75,60){\makebox(0,0)[cc]{$\bullet$}}
\put(100,65){\makebox(0,0)[cc]{$\circ$}}
\put(100,70){\makebox(0,0)[cc]{$\circ$}}
\put(100,75){\makebox(0,0)[cc]{$\circ$}}
\put(100,80){\makebox(0,0)[cc]{$\circ$}}
\put(100,85){\makebox(0,0)[cc]{$\circ$}}
\put(105,55){\makebox(0,0)[cc]{$\circ$}}
\put(105,60){\makebox(0,0)[cc]{$\circ$}}
\put(110,55){\makebox(0,0)[cc]{$\circ$}}
\put(110,60){\makebox(0,0)[cc]{$\circ$}}
\put(115,60){\makebox(0,0)[cc]{$\circ$}}
\put(115,55){\makebox(0,0)[cc]{$\circ$}}
\put(120,60){\makebox(0,0)[cc]{$\circ$}}
\put(120,55){\makebox(0,0)[cc]{$\circ$}}
\put(125,60){\makebox(0,0)[cc]{$\circ$}}
\put(125,55){\makebox(0,0)[cc]{$\circ$}}
\put(125,65){\makebox(0,0)[cc]{$\circ$}}
\put(105,85){\makebox(0,0)[cc]{$\circ$}}
\put(110,80){\makebox(0,0)[cc]{$\circ$}}
\put(115,75){\makebox(0,0)[cc]{$\circ$}}
\put(120,70){\makebox(0,0)[cc]{$\circ$}}
\put(120,65){\makebox(0,0)[cc]{$\bullet$}}
\put(115,65){\makebox(0,0)[cc]{$\bullet$}}
\put(115,70){\makebox(0,0)[cc]{$\bullet$}}
\put(110,70){\makebox(0,0)[cc]{$\bullet$}}
\put(110,65){\makebox(0,0)[cc]{$\bullet$}}
\put(105,65){\makebox(0,0)[cc]{$\bullet$}}
\put(105,70){\makebox(0,0)[cc]{$\bullet$}}
\put(105,75){\makebox(0,0)[cc]{$\bullet$}}

\put(105,80){\makebox(0,0)[cc]{$\bullet$}}
\put(110,75){\makebox(0,0)[cc]{$\bullet$}}
\end{picture}
\caption{Toric description of stable pairs}
\end{figure}
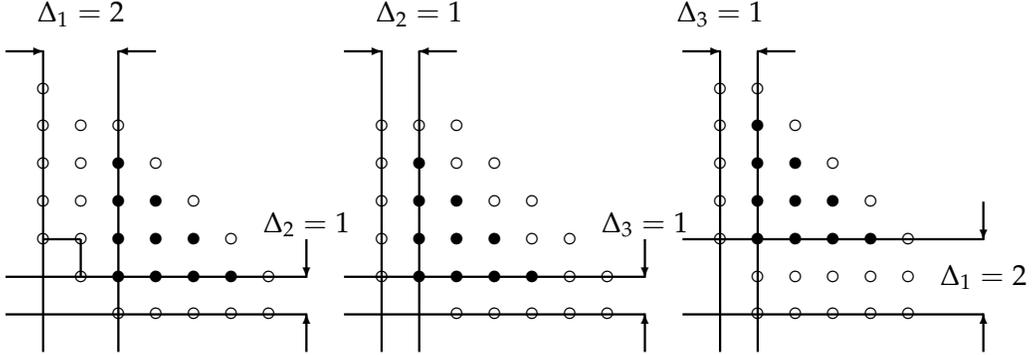
\hfill $\oslash$

In order to get a bijection between the closed points of $\cP_n(\P^2,P)^{\T}$ and the $\T$-equivariant stable pairs with a fixed $\Xi \in \sD(P)$, we need to take a GIT quotient of $C(\Xi,\ell)$ with the action of $\text{SL}(2,\CC)$ induced from the natural action on $\CC^2$. The crucial point is that by a slight modification of the ample line bundle constructed in \cite[Proposition 3.11]{a110}, one can find an ample line bundle $\L_\Xi$ on $C(\Xi,\ell)$ with respect to which the GIT stability matches with the stability of the pairs. In fact in \cite[Proposition 3.11]{a110}, 
an ample line bundle $\boxtimes_{i=1}^{k+e_\Xi}\O(g_i)$ was constructed on $(\P^1)^{k+e_\Xi}$ with respect to which the GIT stability matches with the Gieseker stability of sheaves. 
 If $l=1$, the modified line bundle on $C(\Xi,\ell)$ is taken to be $$(\boxtimes_{i=1}^{k+e_\Xi}\O(g_i))\boxtimes \O(\epsilon),$$  where $\O(\epsilon)$ is a suitable power of $\O(1)$ on the factor of $C(\Xi,\ell)$ corresponding to $l$ (see \eqref{delta-stab}). If $l=0$ then the modification of $\boxtimes_{i=1}^{k+e_\Xi}\O(g_i)$ is made by adjusting the power of $\O(1)$ pulled back from one of the factors of $C(\Xi,\ell)$.  We have proven

\begin{prop}
The Euler characteristic of the moduli space of stable pairs $\cP_n(\P^2,P)$ is given by the following formula
$$\chi(\cP_n(\P^2,P))=\sum_{\Xi,\ell} \chi(C(\Xi,\ell)\sslash_{\L_\Xi}\text{SL}(2,\CC) ),$$ where the sum is over all $\Xi\in \sD(P)$ and the $\Xi$-compatible triples of integers $\ell$. \qed
\end{prop}

\newcommand{\C}{\mathfrak{C}}

We define $\C(\Xi,\ell)=C(\Xi,\ell)\sslash_{\L_\Xi}\text{SL}(2,\CC)$, and denote by $\C(\Xi,\ell)^{sp}$, $\C(\Xi,\ell)^{ss}$, and $\C(\Xi,\ell)^{st}$ the locally closed subspaces of $\C(\Xi,\ell)$ where the corresponding underlying sheaves $\F$ are respectively decomposable, indecomposable strictly semistable, and stable. Given any two $\Xi$-compatible triple of integers $\ell, \ell'$, it is not hard to verify the following properties:

\begin{enumerate} [i)]
\item $\C(\Xi,\ell)^{sp}$ is empty or an isolated point, 
\item $\chi(\C(\Xi,\ell)^{ss})=\chi(\C(\Xi,\ell')^{ss})$; we denote the common value by $\mathsf{c}(\Xi)^{ss}$.
\item $\chi(\C(\Xi,\ell)^{st})/(l(\Xi,\ell)+1)=\chi(\C(\Xi,\ell')^{st})/(l(\Xi,\ell')+1)$; we denote the common value by $\mathsf{c}(\Xi)^{st}$.
\end{enumerate}

In the following proposition we determine the contributions of each of the above items to the Euler characteristc $\chi(\cP_n(\P^2,P))$ ($P$ as in \eqref{Hilb-poly}).

\begin{prop} \label{lem:pair-invariants}
\begin{enumerate}
\item If $\F\cong \I_{Z_1}\oplus \I_{Z_2}$, where $\I_{Z_1}, \I_{Z_2}$ are the ideal sheaves of the $\T$-invariant 0-dimensional subschemes $Z_1,Z_2 \subset \P^2$ such that the Hilbert polynomials of $\I_{Z_1}, \I_{Z_2}$ are equal to $P/2$ then the contribution of $\F$ to the Euler characteristc $\chi(\cP_n(\P^2,P))$ is 
$$\begin{cases}P(n)(P(n)-2)/8 & \text{if } Z_1=Z_2, \\ P(n)^2/4 & \text{if } Z_1\neq Z_2.\end{cases}$$
\item Given  $\Xi\in \sD(P)$  $$\sum_{\ell \text{ is $\Xi$-compatible}}\chi(\C(\Xi,\ell)^{ss})= P(n)\cdot \mathsf{c}(\Xi)^{ss}/2.$$
\item Given  $\Xi\in \sD(P)$  $$\sum_{\ell \text{ is $\Xi$-compatible}}\chi(\C(\Xi,\ell)^{st})= P(n)\cdot \mathsf{c}(\Xi)^{st}.$$
\end{enumerate}
\end{prop}
\begin{proof}
Given $\F$ as in (1), the contribution of $\F$ is equal to the Euler characteristic of the space of $\T$-equivariant sections of $\F(n)$ satisfying the stability condition for the pairs. This is essentially worked out in  \cite[Example 6.1 and 6.2]{a30}. (2) and (3) are proven by fining a weighted number of $\Xi$-compatible triples $\ell$. Each $\ell$ is counted with multiplicity $l(\Xi,\ell)+1$ to account for the Euler characteristics of $\P^1$ and a point (see Example \ref{ex:toric}). 
\end{proof}

\begin{rem} \label{rem:mu}
In the case that $\Xi$ is a non-degenerate $\Delta$-family and $$\Delta_i< \Delta_j +\Delta_k$$ for any pairwise distinct $i,j,k \in \{1,2,3\}$,
it is easy to find $\cc(\Xi)^{ss}$ and $\cc(\Xi)^{st}$. Let $\ell_0$ be a $\Xi$-compatible triple of integers with $l(\Xi,\ell_0)=1$. Since $p_1,p_2,p_3$ are nonzero and pairwise distinct by assumption $C(\Xi,\ell)=(\P^1)^{k+4}$, and we have $$\cc(\Xi)^{st}=2^k \quad \text{and} \quad \cc(\Xi)^{ss}=0.$$ In fact in this case, $\C(\Xi,\ell_0)^{ss}=\emptyset$, and   
$\C(\Xi,\ell_0)^{ss}=(\P^1)^{k+1}$ because by the $\text{SL}(2,\CC)$-action one can fix $p_1,p_2,p_3$ to be respectively $0,1,\infty$ and then $$t,s_1,\dots,s_k\in \P^1$$ can be arbitrary. \hfill $\oslash$
\end{rem}

Using Lemma \ref{lem:wall-crossing} and Proposition \ref{lem:pair-invariants}, we prove the following result evaluating the DT invariants of $X$ corresponding to the rank 2 torsion-free sheaves on $\P^2$:

\begin{thm} \label{thm:ss}
Let $P(m)=m^2+3m+2+b$ where $0\ge b \in \Z$, and let $\M^s(\P^2;P) \subseteq \M(\P^2;P)$ be the open subset of the stable sheaves. \begin{enumerate} \item If $b$ is an odd number then $\DT(X;P)=DT(X;P)=\chi(\M(\P^2;P))$. \item If $b$ is an even number then 
\begin{align*}
\DT(X;P)&=\chi(\hilb^{-b/2}(\P^2))/4-\sum_{\Xi \in \sD(P)}\mathsf{c}(\Xi)^{st}-\frac{1}{2}\sum_{\Xi\in \sD(P)}\mathsf{c}(\Xi)^{ss} \\&=\chi(\hilb^{-b/2}(\P^2))/4-\chi(\M^s(\P^2;P))-\frac{1}{2}\sum_{\Xi \in \sD(P)}\mathsf{c}(\Xi)^{ss}, 
\end{align*}
\end{enumerate}
\end{thm}
\begin{proof} (1) is already proven (see the discussion at the beginning of this section). It also follows from our toric description, as only $\cc^{st}(\Xi)$ is nonzero for any relevant $\Delta$-family data $\Xi$. We now prove (2).
By Lemma \ref{lem:wall-crossing} $$\DT(X;P)=\chi(\hilb^{-b/2}(\P^2))^2\cdot P(n)/8-PI_n(X;P)/P(n).$$ $PI_n(X;P)$ is the sum of the contributions of three types of  $\T$-equivariant semistable sheaves to the Euler characteristics evaluated in Proposition \ref{lem:pair-invariants}. The contribution of the decomposable sheaves is given by:
\begin{align*}&\chi(\hilb^{-b/2}(\P^2))\cdot P(n)(P(n)-2)/8\\&+\chi(\hilb^{-b}(\P^2))(\chi(\hilb^{-b/2}(\P^2))-1)/2\cdot P(n)^2/4.\end{align*}
By Proposition \ref{lem:pair-invariants}, the contribution of the Gieseker stable sheaves is equal to 
$$\sum_{\Xi\in \sD(P)} P(n)\cdot \mathsf{c}(\Xi)^{st}=\chi(\M^s(\P^2;P))\cdot P(n).$$ The equality follows from the fact that Gieseker stable sheaves are simple, and from the definition of the stability of pairs in which the section does not play a role when the underlying sheaf is Gieseker stable. Finally, the contribution of the indecomposable strictly semistable sheaves is  

$$\sum_{\Xi\in \sD(P)} P(n)\cdot \mathsf{c}(\Xi)^{ss}/2.$$ 
Now the formula in the theorem is obtained by adding all these contributions.

This also finishes the proof of Theorem \ref{thm:P2}.
\end{proof}

\begin{rem}\label{rem:muu}  Let $P(m)=m^2+3m+2+b$ be the Hilbert polynomial corresponding to $a=0$ then by Remark \ref{rem:mu} $\chi(\M^{\mu s}(\P^2;P))$ is the number of 9-tuples $$(\Delta_1,\Delta_2,\Delta_3,\pi_1^1,\pi_1^2,\pi_2^1,\pi_2^2,\pi_3^1,\pi_3^2)$$ of positive integers $\Delta_i$ and 2d partitions $\pi^i_j$ such that
$$\Delta_i<\Delta_j+\Delta_k,\quad b=\sum \Delta_i^2/2-\sum_{i,j}\# \pi_i^j.$$
It is not hard to see that 
\begin{align*}  
\sum_{\tiny \begin{array}{c}\tiny \Delta_1,\Delta_2,\Delta_3 > 0 \\ \Delta_i<\Delta_j+\Delta_k\end{array}}q^{(\sum_i \Delta_i)^2/4+\sum_{i<j}\Delta_i\Delta_j}=\sum_{m=1}^\infty\sum_{n=1}^\infty \frac{q^{mn+m+n}}{1-q^{m+n}}.\end{align*}

From this we get 

\begin{equation*}  
\sum_{b\in \Z}\chi(\M^{\mu s}(\P^2;m^2+3m+2+b))q^{-b}=\frac{1}{\prod_{n> 0}(1-q^n)^6}\sum_{m=1}^\infty\sum_{n=1}^\infty \frac{q^{mn+m+n}}{1-q^{m+n}}.\end{equation*}
This is in agreement with \cite[Corollary 4.2]{a110} which uses a slightly different argument.
In order to find $\chi(\M^s(\P^2;P))$ appearing in Theorem \ref{thm:ss} we need to add to the formula above the contribution of the Gieseker stable sheaves which are not $\mu$-stable. This means that one of the indequlaities in Remark \ref{rem:mu} must turn into an equality $\Delta_i+\Delta_j=\Delta_k$ and we must also allow for the degenerate $\Delta$-families.  \hfill $\oslash$

\end{rem}

In the following examples we compute $\DT(X;P)$, $\hat{DT}(X;p)$ and $\chi(\M^s(X;P))$ in the cases $b=0$, $b=-2$, and $b=-4$. 

\begin{example} 
\textbf{$b=0$}. In this case the only semistable sheaf with Hilbert polynomial $P(m)=m^2+3m+2$ is isomorphic $\O\oplus\O$. Therefore, by Proposition \ref{lem:pair-invariants} part (1) we have $PI_n(X;P)=P(n)(P(n)-2)/8$, and hence by Lemma \ref{lem:wall-crossing} and noting that $DT(X;P/2)=1$ we get $$\DT(X;P)=P(n)/8-(P(n)-2)/8=1/4$$ in agreement with the result of Theorem \ref{thm:ss}. We can easily see that $$\hat{DT}(X;P)=0 \quad \text{and} \quad \chi(\M^s(X;P))=0.$$ \hfill $\oslash$
\end{example}
\begin{example}
 $b=-2$. By Proposition \ref{lem:pair-invariants} we have $$PI_n(X;P)=3P(n)^2/4+3P(n)(P(n)-2)/8+12P(n)/2.$$ The first term is the sum of the contributions of $\I_{Z_1}\oplus \I_{Z_2}$ where $Z_1$, $Z_2$ are two distinct $\T$-fixed points of $\P^2$. The second term is the sum of the contributions $\I_Z\oplus \I_Z$ where $Z$ is a $\T$-fixed point of $\P^2$, and the last term is the contributions indecomposable Gieseker semistable sheaves obtained from Table \ref{tab:b=-2}.
 \begin{table}[h]
 \begin{center}
\begin{tabular}{|c||c|c|c|c||c|c|c|}
\hline
& $A$&$(\Delta_1,\Delta_2,\Delta_3)$&$(\pi_1^1,\dots,\pi_3^2)$&$E$&$\cc^{ss}(\Xi)$& $\cc^{st}(\Xi)$ & $\#$\\
\hline
 1 & $-1$ & $(1,1,0)$ &$(\;,\;,\s,\s,\;,\;)$  & $\emptyset$  & $1$ & $0$ & $3$\\
\hline
2 & $-1$ & $(1,1,0)$ &$(\;,\;,\;,\;,\s,\s)$  & $\emptyset$  & $1$ & $0$ & $3$\\
\hline
3 & $-2$ & $(2,1,1)$ &$(\s,\;,\;,\;,\;,\;)$  & $\emptyset$  & $1$ & $0$ & $3$\\
\hline
4 & $-2$ & $(2,1,1)$ &$(\;,\;,\;,\;,\;,\s)$  & $\emptyset$  & $1$ & $0$ & $3$\\
\hline
\end{tabular}
\end{center}
\caption{$\Delta$-family data for the case $b=-2$.}
\label{tab:b=-2}
\end{table}
Columns 2-5 give the $\Delta$-family data $\Xi$ giving rise to Hilbert polynomial $P$. We only consider the cases $\Delta_1\ge \Delta_2 \ge \Delta_3$, and the very last Column we assign a multiplicity to account for the other $\Delta$-families obtained by reindexing $\Delta_i$'s. See the next example in which we provide more details for a few rows of a similar table.

Now using the fact that $DT(X;P/2)=3$ from (\ref{ech}), by Lemma \ref{lem:wall-crossing} we get 
$$\DT(X;P)=9P(n)/8-3P(n)/4-3(P(n)-2)/8-6=-21/4$$ in agreement with the result of Theorem \ref{thm:ss}. We can easily see that $$\hat{DT}(X;P)=-6 \quad \text{and} \quad \chi(\M^s(X;P))=0.$$
\hfill $\oslash$
\end{example}

\begin{example}
$b=-4$.  We have $\chi(\hilb^2(\P^2))=9$ by (\ref{ech}), so decomposable $\T$-equivariant sheaves contributes 9/4 by Theorem \ref{thm:ss}. We summarize the contributions of indecomposable Gieseker semistable $\T$-equivariant sheaves in Table \ref{tab:b=-4}. Columns 2-5 give the $\Delta$-family data $\Xi$ giving rise to the Hilbert polynomial $P$. We only consider the cases $\Delta_1\ge \Delta_2 \ge \Delta_3$, and the very last Column we assign a multiplicity to account for the other $\Delta$-families obtained by reindexing $\Delta_i$s. From Table \ref{tab:b=-4} we get $\sum_{\Xi} \cc^{ss}(\Xi)=216$ and $\sum_{\Xi} \cc^{st}(\Xi)=54$ 
By the formula in Theorem \ref{thm:ss} $$\DT(X;P)=9/4-54-216/2=-639/4.$$ We can easily see that  
 $$\hat{DT}(X;P)=-162 \quad \text{and} \quad \chi(\M^s(X;P))=54.$$ by Corollary \ref{cor:ss}.

In the following we explain the details for four rows of Table \ref{tab:b=-4}. Figure \ref{fig:rs} gives the $\Delta$-family data $\Xi$ corresponding to Row 1 of Table \ref{tab:b=-4}. We fix a $\Xi$-compatible triple $\ell_0$ with $l(\Xi,\ell_0)=1$. We use the following notation $$p_1=p, p_2=q, s_1=r, s_2=s.$$ This data and a point $t \in \P^1$ completely determine the $\T$-equivariant pair $\O(-n)\to \F$. The stability of the pair translates into the following facts:

\begin{enumerate} [i)]
\item $p\neq q$ and $r,s$ cannot both be equal to $p$ or $q$.
\item $r=p, s=q, q\neq t\neq p$ or $r=q, s=p, q\neq t\neq p$ in which case $\F$ is decomposable.   
\item $r=p, s\neq q, t\neq p$ or $r=q, s\neq p, t\neq q$ or $r\neq p, s=q, t\neq q$ or $r\neq q, s=p, t\neq p$ in which case $\F$ is strictly semistable and indecomposable,
\item if $q\neq r\neq p$ and $q\neq s \neq p$, in which case $\F$ is Gieseker stable.
\end{enumerate}
From these facts we can conclude that $\cc(\Xi)^{st}=0$, and $\cc(\Xi)^{ss}=4$ as claimed in the table.

Figure \ref{fig:rr} gives the $\Delta$-family data $\Xi$ corresponding to Row 7 of Table \ref{tab:b=-4}. We fix a $\Xi$-compatible triple $\ell_0$ with $l(\Xi,\ell_0)=1$. We use the following notation $p_1=p, p_2=q, s_1=r.$ This data and a point $t \in \P^1$ completely determine the $\T$-equivariant pair $\O(-n)\to \F$. The stability of the pair translates into 
\begin{enumerate} [i)]
\item $p\neq q$ and $r\neq p$.
\item $r=q, q\neq t\neq p$ in which case $\F$ is decomposable.   
\item $q\neq r\neq p$ and $t\neq p$, in which case $\F$ is Gieseker semistable and indecomposable.
\end{enumerate}
We can conclude that $\cc(\Xi)^{st}=1$, and $\cc(\Xi)^{ss}=0$ as claimed in the table.

Figure \ref{fig:rrr} gives the $\Delta$-family data $\Xi$ corresponding to Row 34 of Table \ref{tab:b=-4}. We fix a $\Xi$-compatible triple $\ell_0$ with $l(\Xi,\ell_0)=1$. We use the following notation $p_1=p, p_2=q, s_1=r.$ This data and a point $t \in \P^1$ completely determine the $\T$-equivariant pair $\O(-n)\to \F$. The stability of the pair translates into $p\neq q $ and $ q\neq r\neq p$ in which case $\F$ is Gieseker stable.
We can conclude that $\cc(\Xi)^{st}=1$, and $\cc(\Xi)^{ss}=0$ as claimed in the table.

Figure \ref{fig:r} gives the $\Delta$-family data $\Xi$ corresponding to Row 68 of Table \ref{tab:b=-4}. We fix a $\Xi$-compatible triple $\ell_0$ with $l(\Xi,\ell_0)=1$. We use the following notation $$p_1=p, p_2=q, s_1=r.$$ This data and a point $t \in \P^1$ completely determine the $\T$-equivariant pair $\O(-n)\to \F$. The stability of the pair translates into the following facts:

\begin{enumerate} [i)]
\item $p\neq q$ and $r\neq q$.
\item $r=p, q\neq t\neq p$ in which case $\F$ is decomposable.   
\item if $q\neq r\neq p$ and $t\neq q$ in which case $\F$ is Gieseker semistable and indecomposable.
\end{enumerate}

From these facts we can conclude that $\cc(\Xi)^{st}=0$, and $\cc(\Xi)^{ss}=1$ as claimed in the table.
\hfill $\oslash$
\end{example}


\begin{figure}[h] 

\ifx\JPicScale\undefined\def\JPicScale{.9}\fi
\unitlength \JPicScale mm
\begin{picture}(135,40)(0,60)
\linethickness{0.3mm}
\put(5,55){\line(0,1){35}}
\linethickness{0.3mm}
\put(10,55){\line(0,1){35}}
\linethickness{0.3mm}
\put(0,60){\line(1,0){40}}
\linethickness{0.3mm}
\put(0,65){\line(1,0){40}}
\linethickness{0.3mm}
\put(50,55){\line(0,1){35}}
\linethickness{0.3mm}
\put(45,60){\line(1,0){35}}
\linethickness{0.3mm}
\put(55,55){\line(0,1){35}}
\linethickness{0.3mm}
\put(50,65){\line(1,0){10}}
\linethickness{0.3mm}
\put(60,60){\line(0,1){5}}
\linethickness{0.3mm}
\put(90,55){\line(0,1){35}}
\linethickness{0.3mm}
\put(85,60){\line(1,0){40}}
\linethickness{0.3mm}
\put(85,65){\line(1,0){40}}
\linethickness{0.3mm}
\put(90,70){\line(1,0){5}}
\linethickness{0.3mm}
\put(95,60){\line(0,1){10}}
\linethickness{0.3mm}
\put(0,90){\line(1,0){5}}
\put(5,90){\vector(1,0){0.12}}
\linethickness{0.3mm}
\put(10,90){\line(1,0){5}}
\put(10,90){\vector(-1,0){0.12}}
\linethickness{0.3mm}
\put(45,90){\line(1,0){5}}
\put(50,90){\vector(1,0){0.12}}
\linethickness{0.3mm}
\put(55,90){\line(1,0){5}}
\put(55,90){\vector(-1,0){0.12}}
\linethickness{0.3mm}
\put(125,65){\line(0,1){5}}
\put(125,65){\vector(0,-1){0.12}}
\linethickness{0.3mm}
\put(125,55){\line(0,1){5}}
\put(125,60){\vector(0,1){0.12}}
\linethickness{0.3mm}
\put(40,65){\line(0,1){5}}
\put(40,65){\vector(0,-1){0.12}}
\linethickness{0.3mm}
\put(40,55){\line(0,1){5}}
\put(40,60){\vector(0,1){0.12}}
\put(5,95){\makebox(0,0)[cc]{$\Delta_{1}=1$}}

\put(40,75){\makebox(0,0)[cc]{$\Delta_{2}=1$}}

\put(52,95){\makebox(0,0)[cc]{$\Delta_{2}=1$}}

\put(125,75){\makebox(0,0)[cc]{$\Delta_{1}=1$}}

\put(8,85){\makebox(0,0)[cc]{p}}

\put(35,63){\makebox(0,0)[cc]{q}}

\put(8,63){\makebox(0,0)[cc]{0}}

\put(53,85){\makebox(0,0)[cc]{q}}

\put(57,63){\makebox(0,0)[cc]{s}}

\put(52,63){\makebox(0,0)[cc]{0}}

\put(110,45){\makebox(0,0)[cc]{}}

\put(92,68){\makebox(0,0)[cc]{r}}
\put(92,63){\makebox(0,0)[cc]{0}}
\put(120,63){\makebox(0,0)[cc]{p}}
\end{picture}

\caption{Row 1 of Table \ref{tab:b=-4}.}
\label{fig:rs}
\end{figure}
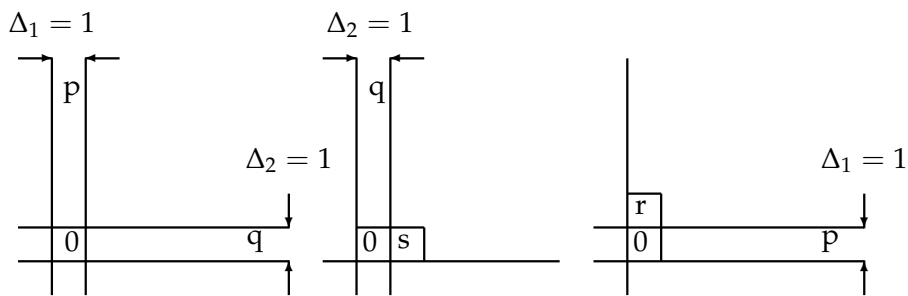

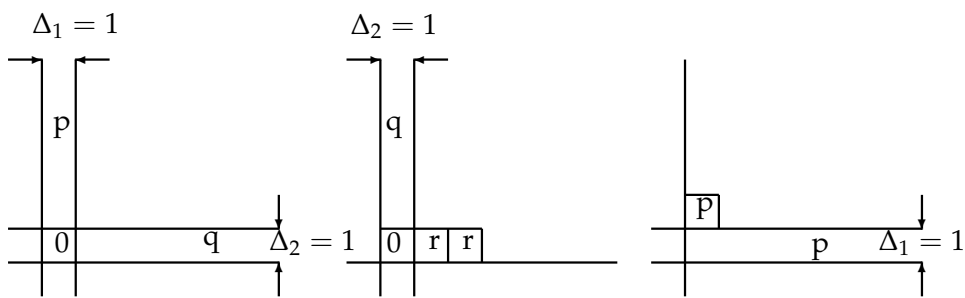
\begin{figure}[h] 
\ifx\JPicScale\undefined\def\JPicScale{.9}\fi
\unitlength \JPicScale mm
\begin{picture}(135,40)(0,55)
\linethickness{0.3mm}
\put(5,50){\line(0,1){35}}
\linethickness{0.3mm}
\put(0,55){\line(1,0){40}}
\linethickness{0.3mm}
\put(0,60){\line(1,0){40}}
\linethickness{0.3mm}
\put(10,50){\line(0,1){35}}
\linethickness{0.3mm}
\put(55,50){\line(0,1){35}}
\linethickness{0.3mm}
\put(50,55){\line(1,0){40}}
\linethickness{0.3mm}
\put(60,50){\line(0,1){35}}
\linethickness{0.3mm}
\put(55,60){\line(1,0){15}}
\linethickness{0.3mm}
\linethickness{0.3mm}
\put(65,55){\line(0,1){5}}
\linethickness{0.3mm}
\put(70,55){\line(0,1){5}}
\linethickness{0.3mm}
\put(100,50){\line(0,1){35}}
\linethickness{0.3mm}
\put(95,55){\line(1,0){40}}
\linethickness{0.3mm}
\put(95,60){\line(1,0){40}}
\put(63,58){\makebox(0,0)[cc]{r}}


\put(68,58){\makebox(0,0)[cc]{r}}

\put(10,90){\makebox(0,0)[cc]{$\Delta_{1}=1$}}

\put(8,75){\makebox(0,0)[cc]{p}}
\put(8,58){\makebox(0,0)[cc]{0}}

\put(45,58){\makebox(0,0)[cc]{$\Delta_{2}=1$}}
\put(30,58){\makebox(0,0)[cc]{q}}

\put(57,90){\makebox(0,0)[cc]{$\Delta_{2}=1$}}
\put(57,75){\makebox(0,0)[cc]{q}}
\put(57,58){\makebox(0,0)[cc]{0}}

\put(135,58){\makebox(0,0)[cc]{$\Delta_{1}=1$}}
\put(120,57){\makebox(0,0)[cc]{p}}

\put(105,60){\line(0,1){5}}
\put(100,65){\line(1,0){5}}
\put(103,63){\makebox(0,0)[cc]{p}}

\linethickness{0.3mm}
\put(0,85){\line(1,0){5}}
\put(5,85){\vector(1,0){0.12}}
\linethickness{0.3mm}
\put(10,85){\line(1,0){5}}
\put(10,85){\vector(-1,0){0.12}}
\linethickness{0.3mm}
\put(40,60){\line(0,1){5}}
\put(40,60){\vector(0,-1){0.12}}
\linethickness{0.3mm}
\put(40,50){\line(0,1){5}}
\put(40,55){\vector(0,1){0.12}}
\linethickness{0.3mm}
\put(50,85){\line(1,0){5}}
\put(55,85){\vector(1,0){0.12}}
\linethickness{0.3mm}
\put(60,85){\line(1,0){5}}
\put(60,85){\vector(-1,0){0.12}}
\linethickness{0.3mm}
\put(135,60){\line(0,1){5}}
\put(135,60){\vector(0,-1){0.12}}
\linethickness{0.3mm}
\put(135,50){\line(0,1){5}}
\put(135,55){\vector(0,1){0.12}}
\end{picture}
\caption{Row 7 of Table \ref{tab:b=-4}.}
\label{fig:rr}
\end{figure}

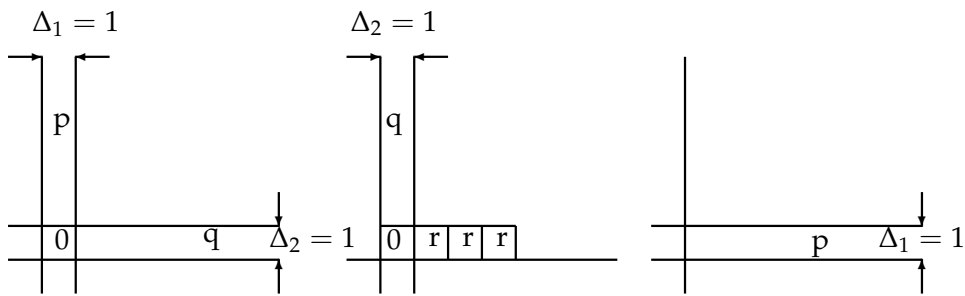
\begin{figure}[h] 
\ifx\JPicScale\undefined\def\JPicScale{.9}\fi
\unitlength \JPicScale mm
\begin{picture}(135,40)(0,55)
\linethickness{0.3mm}
\put(5,50){\line(0,1){35}}
\linethickness{0.3mm}
\put(0,55){\line(1,0){40}}
\linethickness{0.3mm}
\put(0,60){\line(1,0){40}}
\linethickness{0.3mm}
\put(10,50){\line(0,1){35}}
\linethickness{0.3mm}
\put(55,50){\line(0,1){35}}
\linethickness{0.3mm}
\put(50,55){\line(1,0){40}}
\linethickness{0.3mm}
\put(60,50){\line(0,1){35}}
\linethickness{0.3mm}
\put(55,60){\line(1,0){20}}
\linethickness{0.3mm}
\put(75,55){\line(0,1){5}}
\linethickness{0.3mm}
\put(65,55){\line(0,1){5}}
\linethickness{0.3mm}
\put(70,55){\line(0,1){5}}
\linethickness{0.3mm}
\put(100,50){\line(0,1){35}}
\linethickness{0.3mm}
\put(95,55){\line(1,0){40}}
\linethickness{0.3mm}
\put(95,60){\line(1,0){40}}
\put(63,58){\makebox(0,0)[cc]{r}}

\put(73,58){\makebox(0,0)[cc]{r}}

\put(68,58){\makebox(0,0)[cc]{r}}

\put(10,90){\makebox(0,0)[cc]{$\Delta_{1}=1$}}

\put(8,75){\makebox(0,0)[cc]{p}}
\put(8,58){\makebox(0,0)[cc]{0}}

\put(45,58){\makebox(0,0)[cc]{$\Delta_{2}=1$}}
\put(30,58){\makebox(0,0)[cc]{q}}

\put(57,90){\makebox(0,0)[cc]{$\Delta_{2}=1$}}
\put(57,75){\makebox(0,0)[cc]{q}}
\put(57,58){\makebox(0,0)[cc]{0}}

\put(135,58){\makebox(0,0)[cc]{$\Delta_{1}=1$}}
\put(120,57){\makebox(0,0)[cc]{p}}

\linethickness{0.3mm}
\put(0,85){\line(1,0){5}}
\put(5,85){\vector(1,0){0.12}}
\linethickness{0.3mm}
\put(10,85){\line(1,0){5}}
\put(10,85){\vector(-1,0){0.12}}
\linethickness{0.3mm}
\put(40,60){\line(0,1){5}}
\put(40,60){\vector(0,-1){0.12}}
\linethickness{0.3mm}
\put(40,50){\line(0,1){5}}
\put(40,55){\vector(0,1){0.12}}
\linethickness{0.3mm}
\put(50,85){\line(1,0){5}}
\put(55,85){\vector(1,0){0.12}}
\linethickness{0.3mm}
\put(60,85){\line(1,0){5}}
\put(60,85){\vector(-1,0){0.12}}
\linethickness{0.3mm}
\put(135,60){\line(0,1){5}}
\put(135,60){\vector(0,-1){0.12}}
\linethickness{0.3mm}
\put(135,50){\line(0,1){5}}
\put(135,55){\vector(0,1){0.12}}
\end{picture}
\caption{Row 34 of Table \ref{tab:b=-4}.}
\label{fig:rrr}
\end{figure}

\begin{figure}[h] 
\ifx\JPicScale\undefined\def\JPicScale{.9}\fi
\unitlength \JPicScale mm
\begin{picture}(135,40)(0,55)
\linethickness{0.3mm}
\put(5,50){\line(0,1){35}}
\linethickness{0.3mm}
\put(0,55){\line(1,0){40}}
\linethickness{0.3mm}
\put(0,60){\line(1,0){40}}
\linethickness{0.3mm}
\put(15,50){\line(0,1){35}}
\linethickness{0.3mm}
\put(5,65){\line(1,0){10}}
\linethickness{0.3mm}
\put(10,55){\line(0,1){10}}
\linethickness{0.3mm}
\put(15,65){\line(1,0){5}}
\linethickness{0.3mm}
\put(20,55){\line(0,1){10}}
\linethickness{0.3mm}
\put(50,50){\line(0,1){35}}
\linethickness{0.3mm}
\put(45,55){\line(1,0){40}}
\linethickness{0.3mm}
\put(55,60){\line(1,0){30}}
\linethickness{0.3mm}
\put(55,60){\line(0,1){25}}
\linethickness{0.3mm}
\put(95,50){\line(0,1){35}}
\linethickness{0.3mm}
\put(90,55){\line(1,0){40}}
\linethickness{0.3mm}
\put(90,65){\line(1,0){40}}
\linethickness{0.3mm}
\put(100,50){\line(0,1){35}}


\put(18,63){\makebox(0,0)[cc]{r}}
\put(12,58){\makebox(0,0)[cc]{0}}
\put(7,58){\makebox(0,0)[cc]{0}}
\put(7,63){\makebox(0,0)[cc]{0}}
\put(12,63){\makebox(0,0)[cc]{0}}
\put(17,58){\makebox(0,0)[cc]{0}}


\put(10,90){\makebox(0,0)[cc]{$\Delta_{1}=2$}}
\put(10,70){\makebox(0,0)[cc]{p}}

\put(55,90){\makebox(0,0)[cc]{$\Delta_{2}=1$}}
\put(53,70){\makebox(0,0)[cc]{q}}

\put(100,90){\makebox(0,0)[cc]{$\Delta_{3}=1$}}
\put(98,70){\makebox(0,0)[cc]{q}}

\put(35,57){\makebox(0,0)[cc]{$\Delta_{2}=1$}}
\put(25,57){\makebox(0,0)[cc]{q}}

\put(80,57){\makebox(0,0)[cc]{$\Delta_{3}=1$}}
\put(65,57){\makebox(0,0)[cc]{q}}

\put(125,60){\makebox(0,0)[cc]{$\Delta_{3}=2$}}
\put(110,60){\makebox(0,0)[cc]{p}}
\put(97,60){\makebox(0,0)[cc]{0}}

\linethickness{0.3mm}
\put(0,85){\line(1,0){5}}
\put(5,85){\vector(1,0){0.12}}
\linethickness{0.3mm}
\put(15,85){\line(1,0){5}}
\put(15,85){\vector(-1,0){0.12}}
\linethickness{0.3mm}
\put(40,60){\line(0,1){5}}
\put(40,60){\vector(0,-1){0.12}}
\linethickness{0.3mm}
\put(40,50){\line(0,1){5}}
\put(40,55){\vector(0,1){0.12}}
\linethickness{0.3mm}
\put(45,85){\line(1,0){5}}
\put(50,85){\vector(1,0){0.12}}
\linethickness{0.3mm}
\put(55,85){\line(1,0){5}}
\put(55,85){\vector(-1,0){0.12}}
\linethickness{0.3mm}
\put(85,60){\line(0,1){5}}
\put(85,60){\vector(0,-1){0.12}}
\linethickness{0.3mm}
\put(85,50){\line(0,1){5}}
\put(85,55){\vector(0,1){0.12}}
\linethickness{0.3mm}
\put(90,85){\line(1,0){5}}
\put(95,85){\vector(1,0){0.12}}
\linethickness{0.3mm}
\put(100,85){\line(1,0){5}}
\put(100,85){\vector(-1,0){0.12}}
\linethickness{0.3mm}
\put(130,65){\line(0,1){5}}
\put(130,65){\vector(0,-1){0.12}}
\linethickness{0.3mm}
\put(130,50){\line(0,1){5}}
\put(130,55){\vector(0,1){0.12}}
\end{picture}
\caption{Row 68 of Table \ref{tab:b=-4}.}
\label{fig:r}
\end{figure}
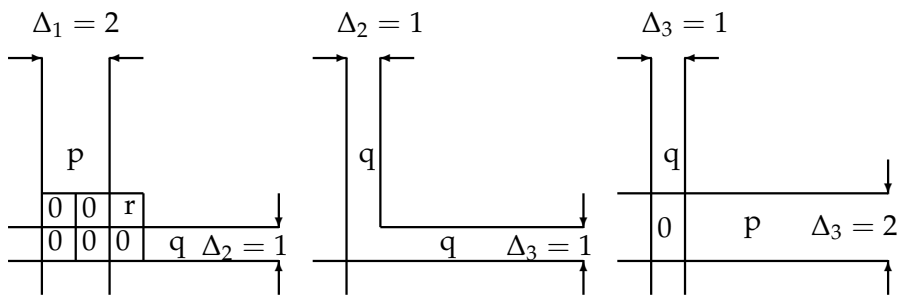


\begin{table} [h]
\begin{tabular}{|c||c|c|c|c||c|c|c|}
\hline
&$A$&$(\Delta_1,\Delta_2,\Delta_3)$&$(\pi_1^1,\dots,\pi_3^2)$&$E$&$\cc^{ss}(\Xi)$& $\cc^{st}(\Xi)$ & $\#$\\
\hline
1& $-1$ & $(1,1,0)$ &$(\;,\;,\s,\s,\s,\s)$  & $\emptyset$  & $4$ & $0$ & 3\\
\hline
2& $-1$ & $(1,1,0)$ &$(\s,\s,\s,\s,\;,\;)$  & $\emptyset$  & $1$ & $0$& 3\\
\hline
3& $-1$ & $(1,1,0)$ &$(\us,\;,\s,\s,\;,\;)$  & $\emptyset$  & $1$ & $0$& 3\\
\hline
4& $-1$ & $(1,1,0)$ &$(\rs,\;,\s,\s,\;,\;)$  & $\emptyset$  & $1$ & $0$& 3\\
\hline
5& $-1$ & $(1,1,0)$ &$(\;,\;,\s,\s,\rs,\;)$  & $\emptyset$  & $1$ & $0$& 3\\
\hline
6& $-1$ & $(1,1,0)$ &$(\;,\s,\s,\rs,\;,\;)$  & $\emptyset$  & $1$ & $0$& 3\\
\hline
7& $-1$ & $(1,1,0)$ &$(\;,\;,\s,\rs,\;,\s)$  & $\emptyset$  & $1$ & $0$& 3\\
\hline
8& $-1$ & $(1,1,0)$ &$(\;,\;,\us,\us,\;,\;)$  & $\emptyset$  & $1$ & $0$& 3\\
\hline
9& $-1$ & $(1,1,0)$ &$(\;,\;,\us,\rs,\;,\;)$  & $\emptyset$  & $1$ & $0$& 3\\
\hline
10& $-1$ & $(1,1,0)$ &$(\s,\;,\us,\s,\;,\;)$  & $\emptyset$  & $1$ & $0$& 3\\
\hline
11& $-1$ & $(1,1,0)$ &$(\;,\;,\us,\s,\s,\;)$  & $\emptyset$  & $1$ & $0$& 3\\
\hline
12& $-1$ & $(1,1,0)$ &$(\s,\;,\s,\s,\s,\;)$  & $\emptyset$  & $1$ & $0$& 3\\
\hline
13& $-1$ & $(1,1,0)$ &$(\;,\s,\s,\s,\s,\;)$  & $\emptyset$  & $1$ & $0$& 3\\
\hline
14& $-1$ & $(1,1,0)$ &$(\s,\s,\;,\;,\s,\s)$  & $\emptyset$  & $1$ & $0$& 3\\
\hline
15& $-1$ & $(1,1,0)$ &$(\;,\rs,\;,\;,\s,\s)$  & $\emptyset$  & $1$ & $0$& 3\\
\hline
16& $-1$ & $(1,1,0)$ &$(\;,\us,\;,\;,\s,\s)$  & $\emptyset$  & $1$ & $0$& 3\\
\hline
17& $-1$ & $(1,1,0)$ &$(\;,\;,\us,\;,\s,\s)$  & $\emptyset$  & $1$ & $0$& 3\\
\hline
18& $-1$ & $(1,1,0)$ &$(\s,\;,\;,\;,\s,\us)$  & $\emptyset$  & $1$ & $0$& 3\\
\hline
19& $-1$ & $(1,1,0)$ &$(\;,\;,\;,\s,\s,\us)$  & $\emptyset$  & $1$ & $0$& 3\\
\hline
20& $-1$ & $(1,1,0)$ &$(\;,\;,\;,\;,\rs,\rs)$  & $\emptyset$  & $1$ & $0$& 3\\
\hline
21& $-1$ & $(1,1,0)$ &$(\;,\;,\;,\;,\rs,\us)$  & $\emptyset$  & $1$ & $0$& 3\\
\hline
22& $-1$ & $(1,1,0)$ &$(\;,\s,\;,\;,\rs,\s)$  & $\emptyset$  & $1$ & $0$& 3\\
\hline
23& $-1$ & $(1,1,0)$ &$(\;,\;,\s,\;,\rs,\s)$  & $\emptyset$  & $1$ & $0$& 3\\
\hline
24& $-1$ & $(1,1,0)$ &$(\;,\s,\s,\;,\s,\s)$  & $\emptyset$  & $1$ & $0$& 3\\
\hline
25& $-1$ & $(1,1,0)$ &$(\s,\;,\s,\;,\s,\s)$  & $\emptyset$  & $1$ & $0$& 3\\
\hline
26& $-1$ & $(1,1,0)$ &$(\rus,\s,\;,\;,\;,\;)$  & $\emptyset$  & $1$ & $0$& 3\\
\hline
27& $-1$ & $(1,1,0)$ &$(\rs,\s,\;,\;,\s,\;)$  & $\emptyset$  & $1$ & $0$& 3\\
\hline
28& $-1$ & $(1,1,0)$ &$(\rs,\s,\;,\;,\;,\s)$  & $\emptyset$  & $1$ & $0$& 3\\
\hline
29& $-1$ & $(1,1,0)$ &$(\rs,\rs,\;,\;,\;,\;)$  & $\emptyset$  & $1$ & $0$& 3\\
\hline
30& $-1$ & $(1,1,0)$ &$(\rs,\s,\s,\;,\;,\;)$  & $\emptyset$  & $1$ & $0$& 3\\
\hline
31& $-1$ & $(1,1,0)$ &$(\rs,\s,\;,\s,\;,\;)$  & $\emptyset$  & $1$ & $0$& 3\\
\hline
32& $-1$ & $(1,1,0)$ &$(\s,\;,\s,\rs,\;,\;)$  & $\emptyset$  & $0$ & $1$& 3\\
\hline
33& $-1$ & $(1,1,0)$ &$(\;,\;,\s,\rs,\s,\;)$  & $\emptyset$  & $0$ & $1$& 3\\
\hline
34& $-1$ & $(1,1,0)$ &$(\;,\;,\s,\rss,\;,\;)$  & $\emptyset$  & $0$ & $1$& 3\\
\hline

\end{tabular}
\end{table}

\begin{table}[h]
\begin{tabular}{|c||c|c|c|c||c|c|c|}

& $A$&$(\Delta_1,\Delta_2,\Delta_3)$&$(\pi_1^1,\dots,\pi_3^2)$&$E$&$\cc^{ss}(\Xi)$& $\cc^{st}(\Xi)$ & $\#$\\
\hline
35& $-1$ & $(1,1,0)$ &$(\;,\s,\;,\;,\s,\us)$  & $\emptyset$  & $0$ & $1$& 3\\
\hline
36& $-1$ & $(1,1,0)$ &$(\;,\;,\s,\;,\s,\us)$  & $\emptyset$  & $0$ & $1$& 3\\
\hline
37& $-1$ & $(1,1,0)$ &$(\;,\;,\;,\;,\s,\uss)$  & $\emptyset$  & $0$ & $1$& 3\\
\hline
38& $-2$ & $(2,1,1)$ &$(\rs,\s,\;,\;,\;,\;)$  & $\emptyset$  & $1$ & $0$ & 3\\
\hline
39& $-2$ & $(2,1,1)$ &$(\rs,\;,\s,\;,\;,\;)$  & $\emptyset$  & $1$ & $0$ & 3\\
\hline
40& $-2$ & $(2,1,1)$ &$(\rs,\;,\;,\s,\;,\;)$  & $\emptyset$  & $1$ & $0$ & 3\\
\hline
41 & $-2$ & $(2,1,1)$ &$(\rs,\;,\;,\;,\s,\;)$  & $\emptyset$  & $1$ & $0$ & 3\\
\hline
42& $-2$ & $(2,1,1)$ &$(\us,\s,\;,\;,\;,\;)$  & $\emptyset$  & $1$ & $0$ & 3\\
\hline
43& $-2$ & $(2,1,1)$ &$(\us,\;,\s,\;,\;,\;)$  & $\emptyset$  & $1$ & $0$ & 3\\
\hline
44& $-2$ & $(2,1,1)$ &$(\us,\;,\;,\s,\;,\;)$  & $\emptyset$  & $1$ & $0$ & 3\\
\hline
45& $-2$ & $(2,1,1)$ &$(\us,\;,\;,\;,\s,\;)$  & $\emptyset$  & $1$ & $0$ & 3\\
\hline
46& $-2$ & $(2,1,1)$ &$(\;,\s,\;,\;,\;,\rs)$  & $\emptyset$  & $1$ & $0$ & 3 \\
\hline
47& $-2$ & $(2,1,1)$ &$(\;,\;,\s,\;,\;,\rs)$  & $\emptyset$  & $1$ & $0$ & 3\\
\hline
48& $-2$ & $(2,1,1)$ &$(\;,\;,\;,\s,\;,\rs)$  & $\emptyset$  & $1$ & $0$ & 3\\
\hline
49& $-2$ & $(2,1,1)$ &$(\;,\;,\;,\;,\s,\rs)$  & $\emptyset$  & $1$ & $0$ & 3\\
\hline
50& $-2$ & $(2,1,1)$ &$(\;,\s,\;,\;,\;,\us)$  & $\emptyset$  & $1$ & $0$ & 3\\
\hline
51& $-2$ & $(2,1,1)$ &$(\;,\;,\s,\;,\;,\us)$  & $\emptyset$  & $1$ & $0$ & 3\\
\hline
52& $-2$ & $(2,1,1)$ &$(\;,\;,\;,\s,\;,\us)$  & $\emptyset$  & $1$ & $0$ & 3\\
\hline
53& $-2$ & $(2,1,1)$ &$(\;,\;,\;,\;,\s,\us)$  & $\emptyset$  & $1$ & $0$ & 3\\
\hline
54& $-2$ & $(2,1,1)$ &$(\s,\s,\;,\;,\;,\s)$  & $\emptyset$  & $1$ & $0$ & 3\\
\hline
55& $-2$ & $(2,1,1)$ &$(\s,\;,\s,\;,\;,\s)$  & $\emptyset$  & $1$ & $0$ & 3\\
\hline
56& $-2$ & $(2,1,1)$ &$(\s,\;,\;,\s,\;,\s)$  & $\emptyset$  & $1$ & $0$ & 3\\
\hline
57& $-2$ & $(2,1,1)$ &$(\s,\;,\;,\;,\s,\s)$  & $\emptyset$  & $1$ & $0$ & 3\\
\hline
58& $-2$ & $(2,1,1)$ &$(\rs,\;,\;,\;,\;,\s)$  & $\emptyset$  & $0$ & $1$ & 3\\
\hline
59& $-2$ & $(2,1,1)$ &$(\us,\;,\;,\;,\;,\s)$  & $\emptyset$  & $0$ & $1$ & 3\\
\hline 
60& $-2$ & $(2,1,1)$ &$(\rus,\;,\;,\;,\;,\;)$  & $\emptyset$  & $0$ & $1$ & 3\\
\hline
61& $-2$ & $(2,1,1)$ &$(\uss,\;,\;,\;,\;,\;)$  & $\emptyset$  & $0$ & $1$ & 3\\
\hline
62& $-2$ & $(2,1,1)$ &$(\rss,\;,\;,\;,\;,\;)$  & $\emptyset$  & $0$ & $1$ & 3\\
\hline
63& $-2$ & $(2,1,1)$ &$(\s,\;,\;,\;,\;,\rs)$  & $\emptyset$  & $0$ & $1$ & 3\\
\hline
64& $-2$ & $(2,1,1)$ &$(\s,\;,\;,\;,\;,\us)$  & $\emptyset$  & $0$ & $1$ & 3\\
\end{tabular}
\end{table}

\begin{table}[h]
\begin{tabular}{|c||c|c|c|c||c|c|c|}

&$A$&$(\Delta_1,\Delta_2,\Delta_3)$&$(\pi_1^1,\dots,\pi_3^2)$&$E$&$\cc^{ss}(\Xi)$& $\cc^{st}(\Xi)$ & $\#$\\
\hline
65& $-2$ & $(2,1,1)$ &$(\;,\;,\;,\;,\;,\rus)$  & $\emptyset$  & $0$ & $1$ & 3\\
\hline
66& $-2$ & $(2,1,1)$ &$(\;,\;,\;,\;,\;,\uss)$  & $\emptyset$  & $0$ & $1$ & 3\\
\hline
67& $-2$ & $(2,1,1)$ &$(\;,\;,\;,\;,\;,\rss)$  & $\emptyset$  & $0$ & $1$ & 3\\
\hline
68& $-2$ & $(2,1,1)$ &$(\rs,\us,\;,\;,\;,\;)$  & $\{(2,3)\}$  & $1$ & $0$& 3\\
\hline
69& $-2$ & $(2,1,1)$ &$(\;,\;,\;,\;,\rs,\us)$  & $\{(2,3)\}$  & $1$ & $0$& 3\\
\hline
70& $-2$ & $(2,2,0)$ &$(\;,\;,\rs,\rs,\;,\;)$  & $\emptyset$  & $1$ & $0$& 3\\
\hline
71& $-2$ & $(2,2,0)$ &$(\s,\;,\rs,\s,\;,\;)$  & $\emptyset$  & $1$ & $0$& 3\\
\hline
72& $-2$ & $(2,2,0)$ &$(\;,\;,\rs,\s,\s,\;)$  & $\emptyset$  & $1$ & $0$& 3\\
\hline
73& $-2$ & $(2,2,0)$ &$(\;,\;,\;,\;,\us,\us)$  & $\emptyset$  & $1$ & $0$& 3\\
\hline
74& $-2$ & $(2,2,0)$ &$(\;,\s,\;,\;,\us,\s)$  & $\emptyset$  & $1$ & $0$& 3\\
\hline
75& $-2$ & $(2,2,0)$ &$(\;,\;,\s,\;,\us,\s)$  & $\emptyset$  & $1$ & $0$& 3\\
\hline
76& $-3$ & $(3,2,1)$ &$(\us,\;,\;,\;,\;,\;)$  & $\emptyset$  & $1$ & $0$ & 6\\
\hline
77& $-3$ & $(3,2,1)$ &$(\;,\;,\;,\;,\;,\us)$  & $\emptyset$  & $1$ & $0$ & 6\\
\hline
78& $-3$ & $(3,2,1)$ &$(\rs,\;,\;,\;,\;,\;)$  & $\emptyset$  & $1$ & $0$ & 6\\
\hline
79& $-3$ & $(3,2,1)$ &$(\;,\;,\;,\;,\;,\rs)$  & $\emptyset$  & $1$ & $0$ & 6\\
\hline
80& $-3$ & $(3,2,1)$ &$(\s,\;,\;,\;,\;,\s)$  & $\emptyset$  & $1$ & $0$ & 6\\
\hline
81& $-3$ & $(2,2,2)$ &$(\s,\;,\;,\;,\;,\;)$  & $\emptyset$  & $0$ & $1$ & 1\\
\hline
82& $-3$ & $(2,2,2)$ &$(\;,\s,\;,\;,\;,\;)$  & $\emptyset$  & $0$ & $1$ & 1\\
\hline
83& $-3$ & $(2,2,2)$ &$(\;,\;,\s,\;,\;,\;)$  & $\emptyset$  & $0$ & $1$ & 1\\
\hline
84& $-3$ & $(2,2,2)$ &$(\;,\;,\;,\s,\;,\;)$  & $\emptyset$  & $0$ & $1$ & 1\\
\hline
85& $-3$ & $(2,2,2)$ &$(\;,\;,\;,\;,\s,\;)$  & $\emptyset$  & $0$ & $1$ & 1\\
\hline
86& $-3$ & $(2,2,2)$ &$(\;,\;,\;,\;,\;,\s)$  & $\emptyset$  & $0$ & $1$ & 1\\
\hline

\end{tabular}
\vspace{5mm}
\caption{$\Delta$-family data for $b=-4$.}
\label{tab:b=-4}
\end{table}



\clearpage

\noindent {\tt{amingh@math.umd.edu}} \hspace{4cm}  {\tt{sheshmani.1@math.osu.edu}}



\begin{thebibliography}{12}

  
\bibitem[Beh09]{a1}
{K.~Behrend}, \emph{{{Donaldson}--{Thomas} invariants via microlocal geometry}},
  {Annals of Math.} \textbf{170, pp. 1307--1338} (2009).

\bibitem[BF97]{a2}
{K.~Behrend and B.~Fantechi}, \emph{{The intrinsic normal cone}}, {Invent.
  Math.} \textbf{128, pp. 45--88} (1997).

\bibitem[ES87]{a126}
{G.~Ellingsrud and S.~A. Str{\o}mme}, \emph{{On the homology of the Hilbert scheme
  of points in the plane}}, {Invent. Math.} \textbf{87 (2), pp.
  343--352} (1987).

\bibitem[GSY07]{a78}
D.~Gaiotto, A.~Strominger, and X.~Yin, \emph{{The M5-Brane Elliptic Genus:
  Modularity and BPS States}}, {J. of high energy phys.} \textbf{8 (070), pp. 1--17} (2007).

\bibitem[GY07]{a95}
{D.~Gaiotto and X.~Yin}, \emph{{Examples of M5-Brane Elliptic Genera}},  {J. of high energy phys.} \textbf{11 (004), pp. 1--12} (2007).


\bibitem[LT98]{a14},
{J.~ Li and G.~ Tian}, \emph{{Virtual moduli cycles and {G}romov-{W}itten invariants of algebraic varieties}}, {J.~ Amer.~ Math.~ Soc.},\textbf{11, pp. 119--174}(1998).

  
\bibitem[GS13]{a123}  
{A.~Gholampour and A.~Sheshmani}, \emph{{Donaldson-Thomas Invariants of 2-dimensional sheaves inside threefolds and modular forms}}, \textbf{arXiv:1309.0050} (2013).  

\bibitem[GST13]{a144}
{A.~Gholampour and A.~Sheshmani and R.~P.~Thomas}, \emph{{Counting curves on surfaces in Calabi-Yau threefolds}}, \textbf{arXiv:1309.0051} (2013).
  
 \bibitem[Got90]{a89}
{L. G\"{o}ttsche}, \emph{{The Betti numbers of the Hilbert scheme of points on a
  smooth projective surface}}, {Math. Ann.} \textbf{286, pp. 193--207} (1990).
 
 \bibitem[OSV04]{a98}
{H.~ Ooguri and A.~ Strominger and C. ~Vafa}, \emph{{Black Hole Attractors and the Topological String}},
{Physics Review D},\textbf{D70:106007}(2004).

\bibitem[HL97]{a9}
\bysame, \emph{The geometry of moduli spaces of sheaves}, {Cambridge University
  press} (1997).

\bibitem[HT10]{a10}
{D.~Huybrechts and R.~P. Thomas}, \emph{Deformation--Obstruction theory for
  complexes via {A}tiyah and {Kodaira}--{Spencer} classes}, {Math. Ann.}
  \textbf{346 (3), pp. 545--569} (2010).


\bibitem[J08]{a58}
{D.~ Joyce}, \emph{{Configurations in abelian categories. {IV}. {Invariants} and changing stability conditions}},{Advances in Mathematics}, \textbf {217 (1), pp. 125--204} (2008).


\bibitem[JS11]{a30}
{D.~Joyce and Y.~Song}, \emph{A theory of generalized Donaldson--Thomas
  invariants}, Memoirs of the AMS \textbf{217 (1020)} (2011).



\bibitem[Kly91]{a119}
A.~A. Klyachko, \emph{{Vector bundles and torsion free sheaves on the
  projective plane}}, {MPI f\"{u}r Mathematik} (1991).

\bibitem[KS08]{a42}
{Maxim Kontsevich and Yan Soibelman},\emph{{Stability structures, motivic {Donaldson}-{Thomas} invariants and cluster transformations}}, \textbf{arXiv:0811.2435} (2008).


\bibitem[Koo08]{a34}
M.~Kool, \emph{Fixed point loci of moduli spaces of sheaves on toric
  varieties}, Advances in Math. \textbf{227, Issue 4, pp. 1700--1755} (2008).

\bibitem[Koo09]{a110}
\bysame, \emph{{Euler Characteristics of Moduli Spaces of Torsion Free Sheaves
  on Toric Surfaces}}, \textbf{arXiv:0906.3393v2} (2009).


\bibitem[Pot93]{a12}
{J.~Le Potier}, \emph{Systemes coherents et structures de niveau}, {Asterisque},
  \textbf{214 (143)} (1993).



\bibitem[Man10]{a135}
{J.~Manschot}, \emph{{The Betti numbers of the moduli space of stable sheaves of
  rank 3 on $\mathbb{P}^2$}}, \textbf{arXiv:1009.1775} (2010).

\bibitem[MNOP06]{a15}
{D.~Maulik, N.~Nekrasov, A.~Okounkov, and R.~Pandharipande},
  \emph{{Gromov}--{Witten} theory and {Donaldson}--{Thomas} theory. {I}}, {Compos.
  Math.} \textbf{11, pp. 1263--1285} (2006).


\bibitem[OSV01]{a100}
{H.~Ooguri, A.~Strominger, and C.~Vafa}, \emph{{Black Hole Attractors and the
  Topological String}}, {Physics Review D.} \textbf{70 (10), pp. 106--119 } (2001).

\bibitem[Per04]{a115}
{M.~Perling}, \emph{{Moduli for Equivariant Vector Bundles of Rank Two on Smooth
  Toric Surfaces}}, {Math. Nachr.} \textbf{265, pp. 87--99} (2004).


 
\bibitem[Tho00]{a20}
{R.~P.~Thomas}, \emph{A holomorphic {Casson} invariant for {Calabi}-{Yau}
  3-folds, and bundles on {K3} fibrations}, {J. Differential Geom.} \textbf{54, pp. 367--438}
  (2000).

\bibitem[VW94]{a136}
{C.~Vafa and E.~Witten}, \emph{{A Strong Coupling Test of S-Duality}}, {J. Nucl.
  Phys.} \textbf{B431, pp. 3--77} (1994).

\end{thebibliography}
\end{document}